\documentclass[oneside,a4paper,11pt,intlimits]{article}

\usepackage{amssymb}
\usepackage{amsmath}
\usepackage{amssymb}
\usepackage{amsthm}
\usepackage{dsfont}
\usepackage{wallpaper}

\usepackage{enumitem}
\setenumerate{label={\rm (\alph{*})}}

\usepackage[frak,operator]{paper_diening}
\usepackage[hidelinks]{hyperref}
\numberwithin{equation}{section} 

\newtheorem{theorem}{Theorem}[section]
\newtheorem{lemma}[theorem]{Lemma}

\newtheorem{corollary}[theorem]{Corollary}

\newtheorem{definition}[theorem]{Definition}

\newtheorem{assumption}[theorem]{Assumption}

\numberwithin{equation}{section}

\providecommand{\supp}{\support}
\providecommand{skp}[2]{{\langle{#1},{#2}\rangle}}

\providecommand{\dx}{\,\mathrm{d}x}
\providecommand{\dy}{\,\mathrm{d}y}
\providecommand{\dz}{\,\mathrm{d}z}
\providecommand{\ds}{\,\mathrm{d}s}
\providecommand{\dt}{\,\mathrm{d}t}

\newcounter{formel}

\providecommand{\meantmp}[2]{#1\langle{#2}#1\rangle}

\begin{document}
  \thispagestyle{empty}
\begin{titlepage}
	\ThisTileWallPaper{\paperwidth}{\paperheight}{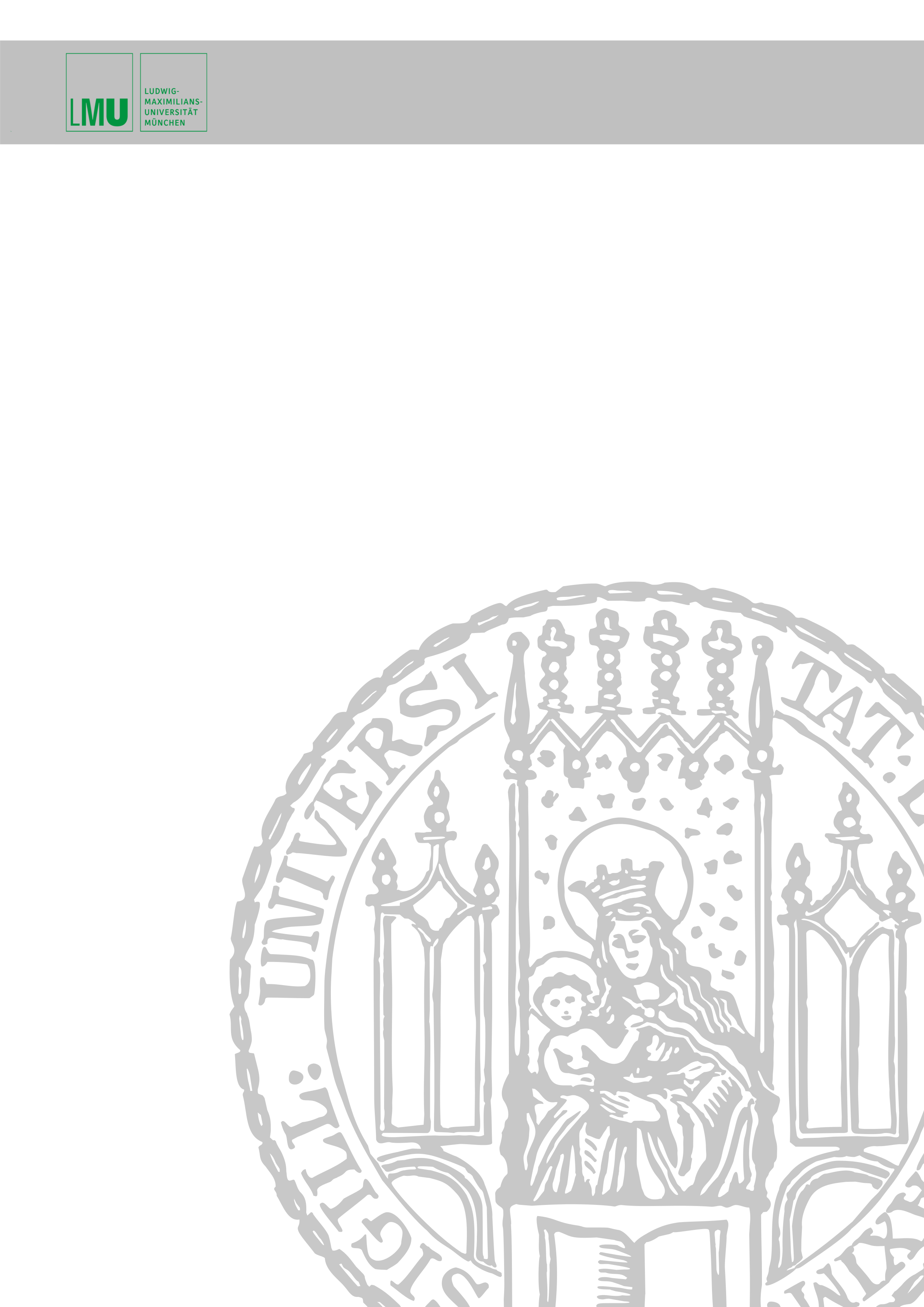}
	\center{	
		\huge{\textbf{The local boundedness of gradients of weak solutions to elliptic and parabolic $\phi$-Laplacian systems}}\\
		\vspace{0.5cm}
		\normalsize{Toni Scharle}\\
		\vspace{1cm}
		\Large{A thesis submitted for the degree of\\
		Master of Science}\\
		\vspace{0.5cm}
		\normalsize{supervised by Prof. Dr. Lars Diening}\\
		\vspace{0.7cm}
		\large{May 2015}\\
		}
		\vspace{\fill}
\end{titlepage}

\newpage

\textbf{Abstract}\\
	In this thesis, a unified approach to prove the boundedness of gradients of solutions to degenerate and singular elliptic and parabolic $\phi$-Laplacian systems is presented. At first, a Cacciopoli-type energy inequality with an additional function $f$ which can  be chosen freely is proven. Then, Di Giorgi's method is applied using level sets  which will lead to $L^\infty$-estimates on the gradient of the weak solution $\nabla \textbf{u}$.
\newpage

\textbf{Acknowledgements}\\
At first I thank my supervisor Prof. Lars Dienig for his support and guidance.\\
I also thank Sebastian Schwarzacher for introducing me to this topic and his support throughout the process of writing of this thesis.\\
Last but not least I am grateful to my parents for their moral and financial support and to my friends and flatmates for making the time of my studies in Munich a truly great time.

\newpage

\tableofcontents
\newpage

\section{Introduction}

In 1900 David Hilbert gave his famous talk \textit{''mathematical problems``}\footnote{"Mathematische Probleme", see \cite{hilbert1900mathematische},translation by the author}  where he described 25 at this moment unsolved problems whose solutions would ''bring an advancement to science`` \footnote{''von deren Behandlung eine F\"orderung der Wissenschaft sich erwarten l\"asst'' see \cite{hilbert1900mathematische}, translation by the author}. \\
The 19th problem reads:\\
\textbf{``Are solutions to regular variational problems always necessarily analytic?''}\\

One subclass of the variational problems Hilbert called regular are those with a given N-function (see section \ref{secNfunctions}) $\phi$ and a domain $\Omega$ where we want to find a function $u\in W^{1,\phi}(\Omega)$ (that means  $\int_\Omega\phi(|\nabla u|)<\infty$)  such that the functional
\begin{equation*}
	\int_\Omega \phi(|\nabla u|)  
\end{equation*}
is minimized under certain boundary conditions.\\
This leads to the elliptic Euler-Lagrange equation (defining $v:=|\nabla u|$)
\begin{equation*}
	\Delta_\phi u:=\text{div}\left(\frac{\phi'(v)}{v}\nabla u\right)=0
\end{equation*}
The best known special case of this is $\phi(t)=t^p$ for $p>1$ where we get the $p$-Laplacian equation:
\begin{equation*}
	\Delta_p u:=\text{div}\left(v^{p-2}\nabla u\right)=0
\end{equation*}

We are now interested in local minimizers of those functionals. This means we are looking for a function $u$ with
$$ \int_{\supp\zeta}\phi(|\nabla u|)\leq\int_{\supp\zeta}\phi(|\nabla u+\nabla\zeta|)$$
for all $\zeta\in C_0^1(\Omega)$. This leads to
$$\int_\omega \frac{\phi'(v)}{v}\nabla u \cdot \nabla \zeta=0 $$
for all $\zeta\in W^{1,\phi}_0(\omega)$ with $\omega\Subset\Omega$.\\
Ennio de Giorgi proved in 1957 (\cite{de1957sulla}) the boundedness of solutions of linear elliptic equations with a truncation method that does not rely on the linearity of the problem and could be easily adopted to prove H\"older continuity of the gradients of those solutions. Independently, Nash got similar results for linear elliptic and parabolic equations in \cite{nash1958continuity} and later Moser proved Harnack estimates for those equations  in \cite{MR0159139}.\\
The boundedness in cases which behave like the $p$-Laplacian equation was given by Uhlenbeck in 1976 in \cite{uhlenbeck1977regularity} in a context of differential forms for $p>2$. The $1<p<2$ case was solved by Acerbi and Fisco in \cite{acerbi1989regularity}. Evans  proved in \cite{evans1982new} qualitative $L^\infty$ bounds by mollification for $p>2$.\\
Marcellini and Papi proved  an estimate on the gradient of solutions to elliptic $\phi$-laplacian systems in \cite{marcellini2006nonlinear}

$$\left(\sup_B v\right)^{2-\beta n}\lesssim \dashint_{2B}\phi(v)+1 $$
where $\beta$ is a $\phi$-dependent constant between $\frac 1 n$ and $\frac 2n$. The restrictions on $\phi$ are so weak that some linear and exponential growth cases are included.\\
Diening, Stroffolini and Verde proved in 2009 (\cite{diening2009everywhere}) under the assumption \ref{mainassumption} which we will also impose on $\phi$ the bound 
$$\sup_B \phi(v)\leq\dashint_{2B} \phi(v)$$ 
which we will get in theorem \ref{statphifinal}. This was further generalized (by substituting assumption \ref{mainassumption} the weaker assumption $c\phi'(t)\leq\phi''(t)t\leq C(1+t)^{\frac \omega 2}\phi'(t)$ for some real $\omega>0$) by Breit, Stroffolini and Verde in \cite{breit2011general}.\\
To get to this point we will use technical tools we develop in section \ref{secNfunctions} to get an energy inequality in section \ref{secStatEnergy}. We will use this to prove the mentioned $L^\infty$-bound with iterated truncations $\chi_{\{v>\gamma\}}$ in section \ref{secStatBound}.\\

We will also look at the parabolic systems. We call a function  $\textbf{u}\in L_{\text{loc}}^\phi(I\times\Omega, \mathbb{R}^m)\cap C_{\text{loc}}(I,L_{\text{loc}}^2(\Omega,\mathbb{R}^m))$ with $v:=|\nabla \textbf{u} |\in L^\phi_{\text{loc}}(I\times\Omega,\mathbb{R})\cap L^2_{\text{loc}}(I\times\Omega,\mathbb{R})$ a local weak solution to $\textbf{u}_t-\Delta_\phi\textbf{u}=0$ on a cylindrical domain $I\times\Omega\subset\mathbb{R}^{1+n}$ if and only if we have for every $\omega\Subset\Omega$ and $[t_1,t_2]\Subset I$

$$\int_{t_1}^{t_2}\int_{\omega} \frac{\phi'(v)}{v} \nabla \textbf{u} \nabla  \boldsymbol{\zeta}\dt\dx=\int_{t_1}^{t_2}\int_{\omega}\textbf{u} \partial_t \boldsymbol{\zeta}\dx\dt-\int_\omega \textbf{u} \boldsymbol{\zeta}\dx\Big\vert_{t_1}^{t_2} $$
for every function $\boldsymbol{\zeta} \in W_{\text{loc}}^{1,2}(I,L_{\text{loc}}^2(\Omega,\mathbb{R}^m))$ with $|\nabla \boldsymbol{\zeta}|\in L_\text{loc}^\phi(I\times\Omega)$.
Equations like this appear for example in the study of non Newtonian fluids and other problems of continuum mechanics. (See \cite{esteban1986equation}.)
For the parabolic $p$-Laplacian systems the most frequently used result is the one obtained by E. DiBenedetto in \cite{dibenedetto1993degenerate} in section VIII.3: If $\textbf{u}$ is a local weak solution to $\textbf{u}_t-\Delta_p\textbf{u}=0$ on a cylinder $I\times\Omega$ he proved, that on a cylinder $Q=J\times B\Subset I\times\Omega$ where $B$ is a ball of radius $R_x$ in $\mathbb{R}^n$ and $J$ an interval of length $R_t=\alpha R_x^2$ and (with $\frac{\nu_r}2=\frac n 2 (p-2)+r$, $r\geq2$):

\begin{align*}
	\sup_Q \frac{v^2}{\alpha}\lesssim &\dashint_{2Q} v^p +\alpha^{\frac{p}{2-p}} \text{ for } p\geq 2\\
	\sup_Q \frac{v^{\frac{\nu_r}{2}}}{\alpha^{\frac{r-p}{2-p}-\frac{n}{2}}}\lesssim&\dashint_{2Q}\frac{v^r}{\alpha^{\frac{p-r}{2-p}}}+\alpha^{\frac{p}{2-p}} \text{ for } p\leq 2
\end{align*}

Although this is an important result, it has a drawback: If the integral on the right hand side tends to zero, the right hand side as a whole is still greater than $\alpha^\frac{p}{2-p}$ and therefore the theorem does not guarantee that if $\norm{v}_p$ and $\norm{v}_2$ go to zero the $L^\infty$ norms also go to zero. The proof itself is not very straightforward and it needs at first a qualitative statement about $v$ being in $L^\infty$ to allow to absorb terms on the left hand side. It starts with the very same Caciopolli-type energy equation we will find in theorem \ref{instatenergy} but uses another function $f$ than we will do. Similar results were obtained earlier by DiBenedetto and Friedman in \cite{dibenedetto1984regularity}.\\
After proving an energy inequality for parabolic $\phi$-Laplacian equation in section \ref{secInstatEnergy} we will get in section \ref{secInstatBound}:

$$\min\left\{\frac {v^{\frac \nu 2}}{\alpha^{\frac{2-n}{n}}},\frac{v^2}{\alpha}\right\}\leq \dashint_{2Q}\frac{v^2}{\alpha}+v^p $$

We see that we do not have to differentiate between the singular and degenerate cases which will allow us to generalize this result to the parabolic $\phi$-Laplacian and whereas DiBenedetto's estimate just provides a constant bound for $v<\alpha^{p-2}$, we just have a switch of exponents. We need $\nu_2>0$ or $p>2-\frac 4 n$ and in this case $r=2$ is the optimal exponent in DiBenedetto's estimate. For larger $r$ there is also an estimate for smaller $p$ provided. Those estimates need a higher integrability for $v$. DiBenedetto's result about the boundedness of gradients of solutions in the case $1<p<2$ was obtained earlier by Choe \cite{choe1991holder}.\\
Acerbi and Mingone proved higher integrability for inhomogeneous $p$-Laplacian systems in \cite{acerbi2007gradient} regaining $\nabla \textbf{u}\in L_{\text{loc}}^q$ if $F\in L^q$ in the inhomogeneity ${\nabla \cdot (|\textbf{F}|^{p-2}\textbf{F})}$.\\

After proving the boundedness of the gradient of parabolic $\phi$-Laplacian systems we could for example apply a result obtained by Liebermann in \cite{lieberman2006holder} where he proved H\"older continuity of gradients of those solutions if there is $L^\infty_{\text{loc}}$ regularity and $\phi$ fulfills assumption \ref{mainassumption}. He proved that, if we have a cylinder $J\times B=:Q\Subset I\times\Omega$ with spacial radius $R_x$, length of $|J|:=R_t=\alpha R_x^2$ and $\alpha=\frac{M}{\phi'(M)}$ and $\sup_Q v\leq M$ we have for a smaller cylinder $Q':=B'\times J'$ with spacial radius $r_x$ and $|J'|=r_t=\alpha r_x^2$ and a positive exponent $\mu$: 

$$\text{osc}_{Q'}|\nabla \textbf{u}|\lesssim M\left(\frac{r_x}{R_x}\right)^\mu $$
This implies H\"older continuity of $\nabla \textbf{u}$. This means, if we have an $\norm{v}_\infty=M_0$ on some cylinder $Q_0$ we have H\"older continuity on all cylinders $Q\subset Q_0$ with $\alpha=\frac{M_0}{\phi'(M_0)}$.

\newpage
\section{N-Functions}\label{secNfunctions}
We use some standard results and definitions from \cite{adams2003sobolev} and \cite{rao1991theory} and start with the definition of an N-Function:
\begin{definition}
	Let $\phi':\mathbb{R}_0^+\rightarrow\mathbb{R}_{0}^+$ be a non-decreasing, left-continuous function with $\phi'(0)=0$, $\phi'(t)>0$ for $t>0$ and $\lim_{t\rightarrow\infty}\phi(t)=\infty$. Then we call the convex function
	$$\phi(t):=\int_0^t \phi'(s) \ds$$
	an N-Function.
\end{definition}
Some common examples are $\phi(t)=t^p$ or $\phi(t)=t \log (t+1)$.\\

Let $\Omega\subset\mathbb{R}^n$ be a domain. The set of measurable functions $\textbf{u}:\Omega\rightarrow \mathbb{R}^m$ with $\int_\Omega \phi(|\textbf{u}|)<\infty$ is called the Orlicz class $L^\phi(\Omega)$. Its span is called Orlicz space  $K^\phi(\Omega)$. On this span we can define the so called Luxemburg norm via
\begin{equation*}
	\norm{\textbf{u}}_\phi=\inf\left\{t>0:\int_\Omega \phi\left(\frac{|\textbf{u}(x)|}{t}\right)\dx \leq 1\right\}
\end{equation*}

\begin{definition}
	For a given N-function we define 
	$$\phi'^{-1}(t)=\inf\{s\geq 0:\phi'(s)>t\} $$
	the complementary N-function via
	$$\phi^\ast(t)=\int_0^t \left(\phi'\right)^{-1}(s)\ds $$
\end{definition}
\noindent It is easy to see that if $\phi$ is strictly increasing, $\phi'^{-1}$ is the true inverse function of $\phi'$.\\
The main reason for this definition is Young's inequality:

$$st\leq \phi(s)+\phi^\ast(t) $$ 
This result is standard and can be found in any textbook about Orlicz spaces, for example \cite{rao1991theory}.\\
With our definition of the Luxemburg norm we also get a H\"older type inequality:

$$\int_\Omega \textbf{f} \textbf{g}\leq 2\norm{\textbf{f}}_\phi\norm{\textbf{g}}_{\phi^\ast} $$

\begin{definition}
 The N-Function $\phi$ is said to fulfill the $\Delta_2$-condition if and only if we have a constant $c$ independent of $t$ such that
 $$\phi(2t)\leq c \phi(t) $$
\end{definition}
As $\phi$ is strictly increasing we can find a constant for every $a>0$ such that $\phi(at)\leq c\phi(t)$ uniformly in $t$. This also implies that the Orlicz-class $L^\phi(\Omega)$ is a vector space and we therefore have $L^\phi(\Omega)=K^\phi(\Omega)$. We will denote the smallest constant $c$ fulfilling $\phi(2t)\leq c\phi(t)$ uniformly in $t$ by $\Delta_2(\phi)$ and for a family of N-Functions $\phi_s$ we will denote $\Delta_2(\{\phi_s\}):=\sup_s\{\Delta_2(\phi_s)\}$.\\ 
If $\Delta_2(\phi)<\infty$ we get
\begin{equation}
	\phi(t)\sim t \phi'(t)
\end{equation}
because of $\frac{\phi(t)}{t}=\frac 1 t\int_0^t\phi'(s)\ds\leq\phi'(t)$ and $\frac{\phi(t)}{t}\geq\frac{\phi(2t)}{t\Delta_2(\phi)}=\frac 1 {t\Delta_2(\phi)} \int_0^{t}\phi'(s)\ds+\frac 1 {t\Delta_2(\phi)} \int_{ t }^{2t}\phi'(s)\ds \geq \frac 1 {\Delta_2(\phi)}\phi'(t)$.\\
If we have $\Delta_2(\phi^\ast)<\infty$, we get 
\begin{align*}
	\phi^\ast(t)\sim t\left(\phi^\ast\right)'(t)=t\left(\phi'\right)^{-1}(t)
\end{align*}
and therefore after setting $t=\phi'(s)$:
\begin{equation}\label{stern}
	\phi^\ast\left(\phi'(s)\right)\sim\phi'(s)s\sim\phi(s)
\end{equation}
If we have $\Delta_2(\{\phi,\phi^\ast\})<\infty$ we also get from Young's inequality, that for every $\epsilon>0$ exists a $c_\epsilon>0$ such that
$$st\leq \epsilon\phi(s)+c_\epsilon\phi^\ast(t) $$

In this thesis we will usually impose a stronger condition than the $\Delta_2$-condition on $\phi$:
\begin{assumption}\label{mainassumption}
	\begin{equation}
		\phi'(t)\sim \phi''(t) t
	\end{equation}
\end{assumption}
We remark that this implies that $\phi$ fulfills the $\Delta_2$-condition.

\begin{definition}
	For a given N-function $\phi$ we define the following functions for $\lambda, t\in\mathbb{R}^+_0$ and $\textbf{Q}\in\mathbb{R}^{n\times m}$:
	\begin{align*}
		\phi_\lambda'(t)&:=\frac{\phi'(\lambda+t)}{\lambda+t}t\\
		\psi'(t)&:=\sqrt{\phi'(t)t}\\
		\textbf{A}(\textbf{Q})&:=\frac{\phi'(|\textbf{Q}|)}{|\textbf{Q}|}\textbf{Q}\\
		\textbf{V}(\textbf{Q})&:=\frac{\psi'(|\textbf{Q}|)}{|\textbf{Q}|}\textbf{Q}
	\end{align*}		
\end{definition}
We will now prove some useful estimates on those quantities.
\begin{theorem}\label{phiestimates}
	With the Definitions as above and $\phi$ with $\Delta_2(\{\phi,\phi^\ast\})<\infty$ fulfilling assumption \ref{mainassumption} we have for all $\textbf{P},\textbf{Q},\textbf{R}\in \mathbb{R}^{n\times m}$:
	\begin{enumerate}
		\item $\partial_{ij}A_{kl}(\textbf{P})=\frac{\phi'(|\textbf{P}|)}{|\textbf{P}|}\left(\tilde{\delta}_{ik}\tilde{\delta}_{jl}-\frac{P_{ij}P_{kl}}{|\textbf{P}|^2}\right)+\phi''(|\textbf{P}|)\frac{P_{ij}P_{kl}}{|\textbf{P}|^2}$ for all $\textbf{P}\in\mathbb{R}^{n\times m}$ where $\tilde{\delta}_{ji}$ is the Kronecker Delta.
		\item $|\textbf{A}(\textbf{P})-\textbf{A}(\textbf{Q})|\lesssim\phi''(|\textbf{P}|+|\textbf{Q}|)|\textbf{P}-\textbf{Q}|$
		\item $\phi''(|\textbf{P}|+|\textbf{Q}|)|\textbf{P}-\textbf{Q}|\sim \phi'_{|\textbf{P}|}(|\textbf{P}-\textbf{Q}|)$
		\item $|\textbf{P}-\textbf{Q}|^2\phi''(|\textbf{P}|+|\textbf{Q}|)\sim \phi_{|\textbf{P}|}(|\textbf{P}-\textbf{Q}|)\sim|\textbf{V}(\textbf{P})-\textbf{V}(\textbf{Q})|^2\sim (\textbf{A}(\textbf{P})-\textbf{A}(\textbf{Q}))(\textbf{P}-\textbf{Q})$
		\item $\phi'_{|\textbf{P}|}(|\textbf{P}-\textbf{Q}|)\lesssim\phi'_{|\textbf{R}|}(|\textbf{P}-\textbf{R}|)+\phi'_{|\textbf{R}|}(|\textbf{Q}-\textbf{R}|)$
	\end{enumerate}
\end{theorem}

\begin{proof}
	\begin{enumerate}
		\item We use $\partial_{ij}P_{kl}=\tilde{\delta}_{ik}\tilde{\delta}_{jl}$ and $\partial_{ij}|\textbf{P}|=\frac{P_{ij}}{|\textbf{P}|}$
		$$\partial_{ij}\left(\frac{\phi'(|\textbf{P}|)}{|\textbf{P}|}P_{kl}\right)= \frac{\phi'(|\textbf{P}|)}{|\textbf{P}|}\tilde{\delta}_{ik}\tilde{\delta}_{jl}+\frac{\phi''(|\textbf{P}|)}{|\textbf{P}|}\frac{P_{ij}}{|\textbf{P}|}P_{kl}-\frac{\phi'(|\textbf{P}|)}{|P|^2}\frac{P_{ij}}{|\textbf{P}|}P_{kl}$$

		\item Define the convex combination $[\textbf{P},\textbf{Q}]_s:=(s\textbf{P}+(1-s)\textbf{Q})$ and  estimate
		\begin{align*}
			 |\textbf{A}(\textbf{P})-\textbf{A}(\textbf{Q})|&=\left\vert\int_0^1 (\nabla \textbf{A})([\textbf{P},\textbf{Q}]_s)(\textbf{P}-\textbf{Q})\ds \right\vert \\
			&\lesssim\int_0^1\frac{\phi'(|[\textbf{P},\textbf{Q}]_s|)}{|[\textbf{P},\textbf{Q}]_s|}\ds|\textbf{P}-\textbf{Q}|\\			&\lesssim\frac{\phi'(|\textbf{P}|+|\textbf{Q}|)}{|\textbf{P}|+|\textbf{Q}|}|\textbf{P}-\textbf{Q}|\\
			&\lesssim\phi''(|\textbf{P}|+|\textbf{Q}|)|\textbf{P}-\textbf{Q}|
		\end{align*}
			The inequality $\int_0^1\frac{\phi'(|[\textbf{P},\textbf{Q}]_s|)}{|[\textbf{P},\textbf{Q}]_s|}\ds\lesssim\frac{\phi'(|\textbf{P}|+|\textbf{Q}|)}{|\textbf{P}|+|\textbf{Q}|}$ is proven in the appendix in lemma \ref{convexintegral}.
		\item We have
		\begin{align*}
			&\phi''(|\textbf{P}|+|\textbf{Q}|)|\textbf{P}-\textbf{Q}|\sim\frac{\phi'(|\textbf{P}|+|\textbf{Q}|)}{|\textbf{P}|+|\textbf{Q}|}|\textbf{P}-\textbf{Q}|\\
			\sim &\frac{\phi'(|\textbf{P}|+|\textbf{P}-\textbf{Q}|)}{|\textbf{P}|+|\textbf{P}-\textbf{Q}|}|\textbf{P}-\textbf{Q}|\sim\phi'_{|\textbf{P}|}(|\textbf{P}-\textbf{Q}|)
		\end{align*}
		where we used the assumption \ref{mainassumption} on $\phi$, the $\Delta_2$-condition and the fact that $|\textbf{P}|+|\textbf{Q}|\sim|\textbf{P}|+|\textbf{P}-\textbf{Q}|$ via ${|\textbf{P}|+|\textbf{Q}|}{=|\textbf{P}|+|\textbf{Q}-\textbf{P}+\textbf{P}|}\leq2 |\textbf{P}|+|\textbf{Q}-\textbf{P}|$ and $|\textbf{P}|+|\textbf{P}-\textbf{Q}|\leq 2|\textbf{P}|+|\textbf{Q}|$.\\
		\item The first similarity follows directly from point (c) and $\phi'(t)t\sim\phi(t)$\\
		For the second similarity we first note that the N-function $\psi$ fulfills assumption \ref{mainassumption} and that we have $\psi''(t)\sim\sqrt{\phi''(t)}$. (Both facts are proven in the appendix in lemma \ref{psi}.) This means we can replace $\phi$ by $\psi$ and $\textbf{A}$ by $\textbf{V}$ in the proof of part (b) and get
		
		$$|\textbf{V}(\textbf{P})-\bf{\textbf{V}}(\textbf{Q})|^2\sim|\textbf{P}-\textbf{Q}|^2\left(\psi''(|\textbf{P}|+|\textbf{Q}|)\right)^2\sim|\textbf{P}-\textbf{Q}|^2\phi''(|\textbf{P}|+|\textbf{Q}|) $$
		
		For the third similarity we use the the compatibility of Frobenius-Norm with Matrix multiplication and point (b) to get:
		\begin{align*}
			|(\textbf{A}(\textbf{P})-\textbf{A}(\textbf{Q}))(\textbf{P}-\textbf{Q})|\leq&|\textbf{A}(\textbf{P})-\textbf{A}(\textbf{Q})|\,|\textbf{P}-\textbf{Q}|\\
			\lesssim&\phi''(|\textbf{P}|+|\textbf{Q}|)|\textbf{P}-\textbf{Q}|^2
		\end{align*}
		
For the other direction we first note that we get for every $\textbf{P},\textbf{B}\in\mathbb{R}^{n\times m}$:
\begin{align*}
	B_{ij}&\left(\partial_{ij}A_{kl}\right)(P)B_{kl}=\frac{\phi'(|\textbf{P}|)}{|\textbf{P}|}\left(|\textbf{B}|^2-\frac{|\textbf{P}\cdot \textbf{B}|^2}{|\textbf{P}|^2}\right)+\phi''(|\textbf{P}|)\frac{|\textbf{P}\cdot \textbf{B}|^2}{|\textbf{P}|^2}\\
	\geq& c\phi''(|\textbf{P}|)\left(|\textbf{B}|^2-\frac{|\textbf{P}\cdot \textbf{B}|^2}{|P|^2}\right)+\phi''(|\textbf{P}|)\frac{|\textbf{\textbf{P}}\cdot \textbf{B}|^2}{|\textbf{P}|^2}\\
	=&(c-\epsilon)\phi''(|\textbf{P}|)\left(|\textbf{B}|^2-\frac{|\textbf{P}\cdot \textbf{B}|^2}{|\textbf{P}|^2}\right)+\epsilon\phi''(|\textbf{P}|)|\textbf{B}|^2\\
	&+(1-\epsilon)\phi''(|\textbf{P}|)\frac{|\textbf{P}\cdot \textbf{B}|^2}{|\textbf{P}|^2}\\
	\geq&\epsilon\phi''(|\textbf{P}|)|\textbf{B}|^2
\end{align*}
where we used point (a) and took $c\in\mathbb{R}^+$ such that $\frac{\phi'(t)}{t}\geq c\phi''(t)$ and $0<\epsilon\leq\min\{1,c\}$.\\
We then estimate $	(\textbf{A}(\textbf{P})-\textbf{A}(\textbf{Q}))(\textbf{P}-\textbf{Q})$ using \ref{convexintegral} and the fact that $\phi$ fulfills assumption \ref{mainassumption}:
\begin{align*}
	(\textbf{A}(\textbf{P})-\textbf{A}(\textbf{Q}))(\textbf{P}-\textbf{Q})=&\int_0^1(\nabla \textbf{A})([\textbf{P},\textbf{S}]_s)(\textbf{P}-\textbf{Q})(\textbf{P}-\textbf{Q})\ds\\
	\gtrsim&\int_0^1\phi''(|[\textbf{P},\textbf{S}]_s)\ds|\textbf{P}-\textbf{Q}|^2\\
	\sim&\phi''(|\textbf{P}|+|\textbf{S}|)|\textbf{P}-\textbf{Q}|^2
\end{align*}

		\item Let us at first assume that $|\textbf{Q}-\textbf{R}|\leq|\textbf{P}-\textbf{R}| $ and therefore $|\textbf{P}-\textbf{Q}|\leq |\textbf{P}-\textbf{R}+\textbf{R}-\textbf{Q}|\leq|\textbf{P}-\textbf{R}|+|\textbf{Q}-\textbf{R}|\leq 2|\textbf{P}-\textbf{R}|$. We also recall that $\Delta_2(\phi_\lambda)$ is bound uniformly in $\lambda$ as proven in lemma \ref{shiftedDelta2} and we therefore get $\phi'_\lambda(2s)\sim\phi'_\lambda(t)$ uniformly in $t$ and $\lambda$. Then we have
		\begin{align*}
			\phi_{|\textbf{P}|}'(|\textbf{P}-\textbf{Q}|)&\leq\phi_{|\textbf{P}|}'(2|\textbf{P}-\textbf{R}|)\\
			&\sim\phi_{|\textbf{P}|}'(|\textbf{P}-\textbf{R}|)\\
			&=\frac{\phi'(|\textbf{P}-\textbf{R}|+|\textbf{P}|)}{|\textbf{P}-\textbf{R}|+|\textbf{P}|}|\textbf{P}-\textbf{R}|\\
			&\sim\frac{\phi'(|\textbf{P}-\textbf{R}|+|\textbf{R}|)}{|\textbf{P}-\textbf{R}|+|\textbf{R}|}|\textbf{P}-\textbf{R}|\\
			&=\phi'_{|\textbf{R}|}(|\textbf{P}-\textbf{R}|)\\
			&\leq \phi'_{|\textbf{R}|}(|\textbf{P}-\textbf{R}|)+\phi'_{|\textbf{R}|}(|\textbf{Q}-\textbf{R}|)
		\end{align*}
		where we used $|\textbf{P}|+|\textbf{P}-\textbf{Q}|=|\textbf{P}-\textbf{Q}+\textbf{Q}|+|\textbf{P}-\textbf{Q}|<2(|\textbf{Q}|+|\textbf{P}-\textbf{Q}|)$ and therefore $|\textbf{P}|+|\textbf{P}-\textbf{Q}|\sim|\textbf{Q}|+|\textbf{P}-\textbf{Q}| $. As we have $\frac{\phi'(|\textbf{P}-\textbf{Q}|+|\textbf{P}|)}{|\textbf{P}-\textbf{Q}|+|\textbf{P}|}|\textbf{P}-\textbf{Q}|\sim\frac{\phi'(|\textbf{P}-\textbf{Q}|+|\textbf{Q}|)}{|\textbf{P}-\textbf{Q}|+|\textbf{Q}|}|\textbf{P}-\textbf{Q}|$ like in the 4th step we can interchange the roles of $|\textbf{P}|$ and $|\textbf{Q}|$.
	\end{enumerate}
\end{proof}

\newpage
\section{Energy estimates}
\subsection{The elliptic case}\label{secStatEnergy}
The main result of this section is the following theorem.
\begin{theorem}[Energy estimate for the elliptic case]\label{statenergy}
 Let $\phi$ be an N-function with $\Delta_2(\{\phi,\phi^\ast\})<\infty$ satisfying the assumption \ref{mainassumption} and let $\textbf{u}\in W^{1,\phi}_{loc}(\Omega,\mathbb{R}^m)$ be a local weak solution to 
 $$\Delta_\phi \textbf{u}=0 $$
 on a domain $\Omega\subset\mathbb{R}^n$ and let $f:\mathbb{R}^+_0\rightarrow\mathbb{R}$ be a non-decreasing, non-negative, bounded, piecewise continuously differentiable function which is constant for large arguments. Define $\textbf{V}(\textbf{Q})=\frac{\sqrt{\phi'(|\textbf{Q}|)}}{|\textbf{Q}|}\textbf{Q}$ as above and denote $v=|\nabla \textbf{u}|$ and let $B\Subset \Omega $ be a ball of radius $R$ and $\eta$ a $C^\infty_0(B)$ function with $0\leq\eta\leq1$.\\
 Then we get for some $q>2$:
 \begin{equation}
  \dashint_B |\nabla \textbf{V}(\nabla \textbf{u})|^2\eta^qf(v)\lesssim\dashint_B\phi(v)|\nabla\eta|^2f(v)
 \end{equation}
\end{theorem}
Before we prove this we restrict the choice of $f$. 
\begin{lemma}\label{statenergylemma}
 The assertion of theorem \ref{statenergy} holds with the additional assumption $f\in C^1$ with $f'(t)\geq 0$ and $f'(t)=0$ for $t$ large enough.
\end{lemma}
\begin{proof}
 We denote $(\tau_{j,h} \textbf{g})(x):=\textbf{g}(x+he_j)-\textbf{g}(x)$, ${(\delta_{j,h} \textbf{g})(x)=\frac 1 h (\tau_{j,h}\textbf{g})(x)}$ and $\delta_h \textbf{g}:=\sum_{j=1}^n{(\delta_{j,h}\textbf{g})e_j}$ and take a $C^\infty_0$ function $\eta$ with $\supp\, \eta \subset B$ and $0\leq\eta\leq 1$. We use the test function ${\zeta:=\delta_{j,-h}(f(|\delta_h \textbf{u}|))\delta_{j,h}\textbf{u}\,\eta^q}$ where we chose $q>2$ such that $\phi(\eta^{q-1}t)\leq\eta^q\phi(t)$ which is possible because of lemma \ref{rausziehen} and we note that $q$ only depends on $\phi$ and not on $\eta$. We get
 \begin{align}
  0=&\langle \textbf{A}(\nabla \textbf{u}),\nabla (\delta_{j,-h}(f(|\delta_h \textbf{u}|)\delta_{j,h}\textbf{u}\,\eta^q))\rangle=	\langle\delta_{j,h}\textbf{A}(\nabla \textbf{u}),\nabla(f(|\delta_h \textbf{u}|)\delta_{j,h}\textbf{u}\,\eta^q)\nonumber\\
	=&\langle \delta_{j,h}\textbf{A}(\nabla \textbf{u}),f'(|\delta_h \textbf{u}|)\nabla|\delta_h \textbf{u}|\delta_{j,h}\textbf{u}\,\eta^q\rangle+\langle \delta_{j,h}\textbf{A}(\nabla \textbf{u}),f(|	\delta_h \textbf{u}|)\delta_{j,h}\nabla \textbf{u}\rangle\nonumber\\
	&+\langle \delta_{j,h}\textbf{A}(\nabla \textbf{u}),f(|\delta_h \textbf{u}|)\delta_{j,h}\textbf{u}\, q\eta^{q-1}\nabla \eta\rangle\nonumber\\
	=:&\text{I}_j+\text{II}_j+\text{III}_j\label{all}
\end{align}
We will at first look at $\text{I}_j$ in \ref{all}. We note  that $|\delta_{j,h} \textbf{u}|f'(|\delta_h \textbf{u}|)\leq|\delta_h \textbf{u}|f'(|\delta_h \textbf{u}|)$ is bounded uniformly in $h$ because of $f'(t)=0$ for large $t$. For the integrand of $\text{I}_j$ this gives
\begin{align}
	&|\delta_{j,h}\textbf{A}(\nabla \textbf{u}) f'(|\delta_h \textbf{u}|)\nabla|\delta_h \textbf{u}|\delta_{j,h}\textbf{u}\eta^q|\nonumber\\
	\leq& |\delta_{j,h}\textbf{A}(\nabla \textbf{u})|\, |\nabla|\delta_h \textbf{u}||\,|f'(|\delta_h \textbf{u}|)\delta_{j,h}\textbf{u}|\nonumber\\
	\lesssim &\frac 1 {h^2}|\tau_{j,h}\textbf{A}(\nabla \textbf{u})|\,|\tau_h\nabla \textbf{u}|\label{Iestimate1}
\end{align}
We now use \ref{phiestimates} (b) and (c)
\begin{align}
	|(\tau_{j,h}\textbf{A})(x)|&=|\textbf{A}((\nabla \textbf{u})(x+h))-\textbf{A}((\nabla \textbf{u})(x))|\nonumber\\
	&\lesssim\phi''(|(\nabla \textbf{u})(x+h)|+|(\nabla \textbf{u})(x)|)|(\tau_{j,h}\nabla \textbf{u})(x)|\nonumber\\
	&\sim\phi'_{|\nabla \textbf{u}|}(|(\tau_{j,h}\nabla \textbf{u})(x)|)\label{tauAestimate}
\end{align}
Using this we return to \ref{Iestimate1} and denote $\max_{j=1,2...,n}|\tau_{j,h}\nabla \textbf{u}|=|\tau_{j_0,h}\nabla \textbf{u}|$ and note that for $n<\infty$ all $p$-norms of $\mathbb{R}^n$ including the supremum norm are equivalent and estimate using the fact that $\phi'_{|\nabla \textbf{u}|}$ is increasing and \ref{phiestimates} (d):
\begin{align}
	\frac 1 {h^2}|(\tau_{j,h}\textbf{A})(x)|\,|\tau_h\nabla \textbf{u}|\sim &\frac 1 {h^2}\phi'_{|\nabla \textbf{u}|}(|(\tau_{j,h}\nabla \textbf{u})(x)|)\,|\tau_h\nabla \textbf{u}|\nonumber\\
	\lesssim &\frac 1 {h^2}\phi'_{|\nabla \textbf{u}|}(|(\tau_{j_0,h}\nabla \textbf{u})(x)|)\,|\tau_{h,j_0}\nabla \textbf{u}|\nonumber\\
	\sim& \frac 1 {h^2}\phi_{|\nabla \textbf{u}|}(|(\tau_{j_0,h}\nabla \textbf{u})(x)|)\nonumber\\
	\sim& \frac 1 {h^2}|\tau_{j_0,h}\textbf{V}(\nabla \textbf{u})(x)|^2\nonumber\\
	\sim& |\delta_h \textbf{V}(\nabla \textbf{u})(x)|^2
\end{align}
As $h\rightarrow 0$, this goes to $|\nabla \textbf{V}(\nabla \textbf{u})|^2$ in $L^2(B)$ since $\textbf{V}(\nabla \textbf{u})\in W^{1,2}_{\text{loc}}(\Omega)$ as proven in Theorem \ref{stattheorem1}. This means we can use a generalized version of the theorem of dominated convergence of Lebesgue which says that if $f_n\rightarrow f$ pointwise almost everywhere and $|f_n|<g_n$ for an $L^1$ convergent sequence $g_n$ we have $\int f_n\rightarrow \int f$.\\
We now need $\delta_{k,h} v\rightarrow \partial_k v$, $\delta_{j,h} (A_{ki}(\nabla u))\rightarrow\partial_{lp}A_{ki}(\nabla \textbf{u})\partial_j\partial_l u_p$ and $\delta_{j,h} u_i\rightarrow\partial_j u_i$. This would be implied by $\nabla \textbf{u} \in W_{\text{loc}}^{2,1}(\Omega)$. It would be possible to show this for a shifted N-function $\phi_\lambda$ with $\lambda>0$ and then we'd have to take the limit $\lambda\rightarrow 0$ in the end like in \cite{diening2009everywhere}. For the sake of clarity and simplicity we will just assume this here. This gives (using the Einstein summation convention and writing $\tilde{\delta}_{ij}$ for the Kronecker-Delta and after a summation over $j$):
\begin{align*}
	\text{I}:=&\sum_{j=1}^m\text{I}_j=\int_B\delta_k v \delta_j\left(A_{ki}(\nabla \textbf{u})\right)\delta_j u_i f'(|\delta_h \textbf{u}|)\dx\\
	\rightarrow&\int_B\partial_k v \left(\partial_{lp}A_{ki}\right)(\nabla \textbf{u})\partial_j\partial_l u_p \partial_j u_i f'(v)\dx\\
	=&\int_B\partial_k v \left(\frac{\phi'(v)}{v}\left(\tilde{\delta}_{lk}\tilde{\delta}_{pi}-\frac{\partial_lu_p\partial_ku_i}{v^2}\right)+\phi''(v)\frac{\partial_lu_p\,\partial_ku_i}{v^2}\right)\partial_j\partial_lu_p\,\partial_ju_if'(v)\dx\\
	=&\int_B\frac{\phi'(v)}{v}\left(\partial_l v \,\partial_j\partial_l u_i\,\partial_ju_i-\frac{\partial_k v\,\partial_ku_i\;\partial_j\partial_lu_p\,\partial_lu_p\partial_ku_i}{v^2}\right)f'(v)\dx\\
	&+\int_B\phi''(v)\frac{\partial_k v\,\partial_ku_i\;\partial_j\partial_lu_p\,\partial_lu_p\partial_ku_i}{v^2}f'(v)\dx\\
	=&\int_B\left(\frac{\phi'(v)}{v}\left(|\nabla v|^2-\frac{|\nabla v\cdot\nabla \textbf{u}|^2}{v^2}\right)+\phi''(v)\frac{|\nabla v\cdot\nabla \textbf{u}|^2}{v^2}\right)f'(v)\dx
\end{align*}

Since we have $f'\geq 0$ and $|\nabla v\cdot\nabla \textbf{u}|^2\leq v^2|\nabla v|^2$ because of the Cauchy-Schwartz inequality, we get
\begin{equation}\label{I}
	\lim_{h\rightarrow 0} \text{I}\geq 0
\end{equation}
To estimate $\text{II}_j$ we apply theorem \ref{phiestimates}(d) and get like in \cite{diening2008fractional}:
\begin{align*}
	(\tau_{j,h} &\textbf{A}(\nabla \textbf{u}))(x)\cdot(\tau_{j,h}\nabla \textbf{u})(x)\\
	&=\left(\textbf{A}(\nabla \textbf{u}(x+h))-\textbf{A}((\nabla \textbf{u})(x))\right)\cdot (\tau_{j,h}\nabla u)(x)\\
	&\sim |(\tau_{j,h} \textbf{V}(\nabla \textbf{u}))(x)|^2
\end{align*}
Dividing by $h^2$ gives
\begin{equation*}
	(\delta_{j,h} \textbf{A}(\nabla \textbf{u}))(x)\cdot(\delta_{j,h}\delta \textbf{u})(x)\sim |(\delta_{j,h} \textbf{V}(\nabla \textbf{u}))(x)|^2
\end{equation*}
Using this we get
\begin{equation}\label{II}
	\text{II}_j=\langle\delta_{j,h} \textbf{A}(\nabla \textbf{u}),f(|\delta_h \textbf{u}|)\delta_{j,h} \nabla \textbf{u}\,\eta^q\rangle\sim\dashint_B|\delta_{j,h} \textbf{V}(\nabla \textbf{u})|^2f(|\delta_h \textbf{u}|)\eta^q
\end{equation}

 We use \ref{tauAestimate} to estimate $\text{III}_j$ and note that
\begin{equation*}
	|(\delta_{j,h} \textbf{u})(x)|=\left\vert\dashint_0^h(\partial_j\textbf{u})(x+se_j)\ds\right\vert\leq\dashint_0^h|(\nabla \textbf{u}\circ T_{se_j})(x)|\ds
\end{equation*}
This gives

\begin{align}
	|\text{III}_j|=&|\langle\delta_{j,h}\textbf{A}(\nabla \textbf{u}),f(|\delta_h \textbf{u}|)\delta_{j,h} \textbf{u}\,q\eta^{q-1}\nabla \eta\rangle|\nonumber\\
	\lesssim&\frac 1 {h^2}\dashint_B\dashint_0^h\eta^{q-1}\phi'_{|\nabla \textbf{u}|}(|\tau_{j,h}\nabla \textbf{u}|)|\nabla \textbf{u}\circ T_{se_j}|\,h|\nabla \eta| f(|\delta_h \textbf{u}|)\ds\label{IIIbeg}
\end{align}
We now estimate the integrand using theorem \ref{phiestimates} (e), Young's inequality, equation \ref{stern}, $h|\nabla\eta|\leq 1$ with Lemma \ref{shiftedrausziehen} and theorem \ref{phiestimates} (d):
\begin{align}
	&\eta^{q-1}\phi'_{|\nabla \textbf{u}|}(|\tau_h\nabla \textbf{u}|)|\nabla \textbf{u}\circ T_{se_j}|\,h|\nabla \eta|\nonumber\\
	\lesssim&\eta^{q-1}\left(\phi'_{|\nabla \textbf{u}\circ T_{se_j}|}(|\tau_{j,h-s}\nabla \textbf{u}\circ T_{se_j}|)+\phi'_{|\nabla \textbf{u}\circ T_{se_j}|}(|\tau_{s}\nabla \textbf{u}|)\right)h|\nabla \eta|\,|\nabla \textbf{u}\circ T_{se_j}|\nonumber\\
	\leq&\epsilon\left(\phi_{|\nabla \textbf{u}\circ T_{se_j}|}\right)^{\ast}\left(\eta^{q-1}\phi'_{|\nabla \textbf{u}\circ T_{se_j}|}(|\tau_{j,h-s}\nabla \textbf{u}\circ T_{se_j}|)\right)\nonumber\\
	&+\epsilon\left(\phi_{|\nabla \textbf{u}\circ T_{se_j}|}\right)^{\ast}\left(\eta^{q-1}\phi'_{|\nabla \textbf{u}\circ T_{se_j}|}(|\tau_{s}\nabla \textbf{u}|)\right)\nonumber\\
	&+c_{\epsilon}\phi_{|\nabla \textbf{u}\circ T_{se_j}|}\left(h|\nabla \eta|\,|\nabla \textbf{u}\circ T_{se_j}|\right)\nonumber\\
	\lesssim&\epsilon\eta^q\left(\phi_{|\nabla \textbf{u}\circ T_{se_j}|}\right)^{\ast}\left(\phi'_{|\nabla \textbf{u}\circ T_{se_j}|}(|\tau_{j,h-s}\nabla \textbf{u}\circ T_{se_j}|)\right)\nonumber\\
	&+\epsilon\eta^q\left(\phi_{|\nabla \textbf{u}\circ T_{se_j}|}\right)^{\ast}\left(\phi'_{|\nabla \textbf{u}\circ T_{se_j}|}(|\tau_{s}\nabla \textbf{u}|)\right)\nonumber\\
	&+c_\epsilon h^2|\nabla \eta|^2\phi\left(|\nabla \textbf{u}\circ T_{se_j}|\right)\nonumber\\
	\lesssim& \epsilon \eta^q \phi_{|\nabla \textbf{u}\circ T_{se_j}|}\left(|\tau_{j,h-s}\nabla \textbf{u}\circ T_{se_j}|\right)
	+\epsilon\phi_{|\nabla \textbf{u}\circ T_{se_j}|}\left(|\tau_{s}\nabla \textbf{u}|\right)\nonumber\\
	&+c_\epsilon h^2|\nabla \eta|^2\phi\left(|\nabla \textbf{u}\circ T_{se_j}|\right)\nonumber\\
	\sim&\epsilon\eta^q|\tau_{j,h-s}\textbf{V}(\nabla \textbf{u})\circ T_{se_j}|^2+\epsilon\eta^q|\tau_{j,s} V(\nabla \textbf{u})|^2+c_\epsilon h^2|\nabla \eta|^2\phi\left(|\nabla \textbf{u}\circ T_{se_j}|\right)\label{phiestimate}
\end{align}
Putting this in \ref{IIIbeg} we get
\begin{align}
	|\text{III}_j|=&|\langle\delta_{j,h}\textbf{A}(\nabla \textbf{u}),f(|\delta_h \textbf{u}|)\delta_{j,h} \textbf{u}\;q\eta^{q-1}\nabla \eta\rangle|\nonumber\\
	\lesssim&\frac {\epsilon} {h^2}\dashint_B\dashint_0^h \left\vert\tau_{j,h-s}\textbf{V}(\nabla \textbf{u})\circ T_{se_j}\right\vert^2f(|\delta_h \textbf{u}|)\ds\nonumber\\
	&+\frac {\epsilon} {h^2}\dashint_B\dashint_0^h \left\vert\tau_{j,s} \textbf{V}(\nabla \textbf{u})\right\vert^2f(|\delta_h \textbf{u}|)\ds\nonumber\\
	&+c_\epsilon \dashint_B\dashint_0^h\phi(|\nabla \textbf{u} \circ T_{se_j}|)|\nabla \eta|^2f(|\delta_h \textbf{u}|)\ds\label{III}
\end{align}
Putting \ref{II} and \ref{III} in \ref{all} we get after a summation over $j$
\begin{align}
	\text{I}+\text{I}':=&\text{I}+\dashint_B|\delta_h \textbf{V}(\nabla \textbf{u})|^2 f(|\delta_h \textbf{u}|)\eta^q\nonumber\\
	\lesssim& \epsilon \sum_{j=1}^m\dashint_B\dashint_0^h \left\vert\frac{\tau_{j,h-s}\textbf{V}(\nabla \textbf{u})\circ T_{se_j}}{h}\right\vert^2f(|\delta_h \textbf{u}|)\eta^q\ds\nonumber\\
	&+ \epsilon \sum_{j=1}^m\dashint_B\dashint_0^h \left\vert\frac{\tau_{j,s} \textbf{V}(\nabla u)}{h}\right\vert^2f(|\delta_h \textbf{u}|)\eta^q\ds\nonumber\\
	&+c_\epsilon  \sum_{j=1}^m\dashint_B\dashint_0^h\phi(|\nabla \textbf{u} \circ T_{se_j}|)|\nabla \eta|^2f(|\delta_h \textbf{u}|)\ds\nonumber\\
	&=:\epsilon\sum_{j=1}^m\text{II}'_j+\epsilon\sum_{j=1}^m\text{III}'_j+c_\epsilon\sum_{j=1}^m\text{IV}'_j
	\label{discreteEnergy}
\end{align}

We now want to take the limit $h\rightarrow 0$ in \ref{discreteEnergy} and know from equation \ref{I} that $\lim_{h\rightarrow 0}\text{I}\geq 0$ and note that $\textbf{V}(\nabla \textbf{u})\in W^{1,2}_{\text{loc}}(\Omega)$ as proved in theorem \ref{stattheorem1}. This means we have $\delta \textbf{V}(\nabla \textbf{u})\rightarrow \nabla \textbf{V}(\nabla \textbf{u})$ in $L^2(B)$. Since $\textbf{u}\in W^{1,\phi}_{\text{loc}}(\Omega)$ we also have $\delta_h \textbf{u}\rightarrow \nabla \textbf{u}$ and therefore $f(|\delta_h\textbf{u}|)\rightarrow f(v)$ pointwise almost everywhere for a subsequence and as $\eta\in C^{\infty}_0(B)$ $\eta^q$ is uniformly continuous.\\
For $\text{I}'$ this means (passing to this subsequence)

\begin{align*}
	&\left\vert\dashint_B|\delta_h \textbf{V}(\nabla \textbf{u})|^2 f(|\delta_h \textbf{u}|)\eta^q-\dashint_B|\nabla \textbf{V}(\nabla \textbf{u})|^2 f(v)\eta^q\right\vert\\
	\leq &\dashint_B \left\vert|\delta_h \textbf{V}(\nabla \textbf{u})|^2-|\nabla \textbf{V}(\nabla \textbf{u})|^2\right\vert\, f(|\delta_h \textbf{u}|)\eta^q\\
	&+\dashint_B |\nabla \textbf{V}(\nabla \textbf{u})|^2\,|f(|\delta_h \textbf{u})- f(v)|\eta^q\\
	=:&\text{I}'_1+\text{I}'_2
\end{align*}
Since $f(|\delta_h \textbf{u}|)\eta^q\leq \norm{f}_\infty$ and  $\delta \textbf{V}(\nabla \textbf{u})\rightarrow \nabla \textbf{V}(\nabla \textbf{u})$ in $L^2(B)$ $\text{I}'_1$ tends to zero. For the integrand in $\text{I}'_2$ we have the dominating function $\norm{f}_\infty\,|\nabla \textbf{V}(\nabla \textbf{u})|^2$ and this summand also goes to zero by dominated convergence as  $f(|\delta_\textbf{u}|)\rightarrow f(v)$ pointwise almost everywhere. In total this gives
\begin{equation}\label{IKonvergenz}
	\text{I}'\rightarrow \dashint_B|\nabla \textbf{V}(\nabla \textbf{u})|^2 f(v)\eta^q
\end{equation}

We now look at $\text{IV}'_j$ and use the theorem of Fubini-Tonelli:
\begin{align*}
	&\left\vert\dashint_B\dashint_0^h\phi(|\nabla \textbf{u} \circ T_{se_j}|)|\nabla \eta|^2f(|\delta_h \textbf{u}|)\ds-\dashint_B\dashint_0^h\phi(|\nabla \textbf{u}|)|\nabla \eta|^2f(v)\ds\right\vert\\
	\leq & \dashint_B\dashint_0^h\left\vert\left(\phi(|\nabla \textbf{u} \circ T_{se_j}|)-\phi(|\nabla \textbf{u}| )\right)|\nabla \eta|^2f(|\delta_h \textbf{u}|)\right\vert\ds\\
	+&\dashint_B|\phi(|\nabla \textbf{u} |)|\nabla \eta|^2\left(f(|\delta_h \textbf{u}|)-f(v)\right)|\\
	\lesssim&\norm{f|\nabla\eta|^2}_\infty\dashint_0^h\dashint_B|\phi(|\nabla \textbf{u} \circ T_{se_j}|)-\phi(|\nabla \textbf{u} |)|\ds\\
	&+\dashint_B\phi(|\nabla \textbf{u}| )|f(|\delta_h \textbf{u}|)-f(v)|\,|\nabla\eta|^2\\
	=:&\text{IV}'_{j,1}+\text{IV}'_{j,2}
\end{align*}
 \sloppy To show $\text{IV}'_{j,2}\rightarrow 0$ we use dominated convergence with the dominant $\phi(v)\norm{f|\nabla\eta|^2}_\infty$ and $f(|\delta_h\textbf{u}|)\rightarrow f(v)$ pointwise almost everywhere for a subsequence as above. To estimate $\text{IV}'_{j,1}$ we use the $L^\phi$-continuity of translations and  the third implication in lemma \ref{phiconvergence} and observe that 
$$ g:\; s\mapsto \dashint_B\left\vert\phi(|\nabla \textbf{u} \circ T_{se_j}|)-\phi(|\nabla \textbf{u} |)\right\vert $$
is a continuous function with $g(0)=0$. But with the fundamental theorem of calculus we have 
$$\lim_{h\rightarrow 0}\frac 1 h \int_0^h g(s)\ds=\frac{\text{d}}{\text{d}h}\int_0^h g(s)\ds=g(0) =0$$
and therefore $\text{IV}'_{j,1}\rightarrow 0$ and after choosing a subsequence we get 
\begin{equation}\label{IVKonvergenz}
	\text{IV}'_j\rightarrow \dashint_B\phi(|\nabla \textbf{u}|)|\nabla \eta|^2f(v) 
\end{equation}
We now want to estimate $\text{III}'_j$ (from \ref{discreteEnergy}) and observe using $h>s$:
$$\text{III}'_j=\dashint_B\dashint_0^h \left\vert\frac{\tau_{j,s} \textbf{V}(\nabla \textbf{u})}{h}\right\vert^2f(|\delta_h \textbf{u}|)\eta^q\ds\leq\dashint_0^h \dashint_B\left\vert\delta_{j,s} \textbf{V}(\nabla \textbf{u})\right\vert^2f(|\delta_h \textbf{u}|)\eta^q\ds=:\text{III}''_j $$
We estimate this term:

\begin{align*}
	&\left\vert\dashint_0^h \dashint_B|\delta_{j,s} \textbf{V}(\nabla \textbf{u})|^2f(|\delta_h \textbf{u}|)\eta^q\ds-\dashint_0^h \dashint_B|\,|\partial_j \textbf{V}(\nabla \textbf{u})|^2f(v)\eta^q\ds\right\vert\\
	\leq&\norm{f}_\infty\dashint_0^h \dashint_B\left\vert\delta_{j,s} \textbf{V}(\nabla \textbf{u})|^2-|\partial_j \textbf{V}(\nabla \textbf{u})|^2 \right\vert\ds\\
	&+\dashint_0^h \dashint_B|\partial_j \textbf{V}(\nabla \textbf{u})|^2|f(|\delta_h \textbf{u}|)-f(v)|\eta^q\ds|\\
	=:&\text{III}''_{j,1}+\text{III}''_{j,2}
\end{align*}
We have $\text{III}''_{j,2}\rightarrow 0$ for $h\rightarrow 0$ in a subsequence as we had $\text{IV}'_{j,2}\rightarrow 0$ as the integrand is bounded by $\norm{f}_\infty|\partial_j \textbf{V}(\nabla \textbf{u})|^2\in L^1(B)$ and we can use dominated convergence.\\
To estimate $\text{III}''_{j,1}$ we note that if $w_n\rightarrow w$ in $L^2$ also $\norm{w_n}_{L^2}\rightarrow\norm{w}_{L^2}$ and we get using $\textbf{V}(\nabla \textbf{u})\in W^{1,2}_{\text{loc}}(\Omega)$:
$$s\mapsto \dashint_B (|\delta_{j,s} \textbf{V}(\nabla \textbf{u})|^2-|\partial_j \textbf{V}(\nabla \textbf{u})|^2)$$
is also a continuous function which is $0$ at $s=0$ and using the same arguments we used for $\text{IV}'_j$ we get $\text{III}''_{j,1}\rightarrow 0$ and therefore 
\begin{equation}\label{IIIKonvergenz}
	\text{III}'_j\leq\text{III}''_j\rightarrow \dashint_B|\partial_j \textbf{V}(\nabla \textbf{u})|^2f(v)\eta^q
\end{equation}
For $\text{II}_j'$ in \ref{discreteEnergy} we first use the invariance of the Lebesgue measure under translations. We also chose $h$ small enough that the closure of the ball $B'$ with the same center as $B$ and radius $r+h$ is contained in $\Omega$ which is possible since $B\Subset\Omega$ and get
\begin{align*}
	|B|\text{II}_J'=&|B|\dashint_B\dashint_0^h \left\vert\frac{\tau_{j,h-s}\textbf{V}(\nabla \textbf{u})\circ T_{se_j}}{h}\right\vert^2f(|\delta_h\textbf{u}|)\eta^q\ds\\
	\leq&\dashint_0^h \int_{B'} \left\vert\frac{\tau_{j,h-s}\textbf{V}(\nabla \textbf{u})}{h-s}\right\vert^2\left(\left(\eta^q f(|\delta_h \textbf{u}|)\right)\circ T_{-se_j}\right)\ds\\
	=&\dashint_0^h \int_{B'} \left\vert\frac{\tau_{s}\textbf{V}(\nabla \textbf{u})}{s}\right\vert^2\left(\left(\eta^q f(|\delta_h \textbf{u}|)\right)\circ T_{(s-h)e_j}\right)\ds=:\text{II}_j''
\end{align*}

We then have

\begin{align*}
	&\left\vert\dashint_0^h \int_{B'} |\delta_{s,j}\textbf{V}(\nabla \textbf{u})|^2\left(\left(\eta^q f(|\delta_h \textbf{u}|)\right)\circ T_{(s-h)e_j}\right)-|\partial_j \textbf{V}(\nabla \textbf{u})|^2f(v)\eta^q\ds\right\vert\\
	\leq&\dashint_0^h \int_{B'} \left\vert|\delta_{s,j}\textbf{V}(\nabla \textbf{u})|^2-|\partial_j \textbf{V}(\nabla \textbf{u})|^2\right\vert\,\left(\eta^q f(|\delta_h \textbf{u}|)\right)\circ T_{(s-h)e_j}\ds\\
	&+\dashint_0^h \int_{B'}|\partial_j \textbf{V}(\nabla \textbf{u})|^2\,\left\vert\left(\eta^q f(|\delta_h \textbf{u}|)\right)\circ T_{(s-h)e_j}-\left(\eta^q f(|\delta_h \textbf{u}|)\right)\circ T_{-he_j}\right\vert\ds\\
	&+\dashint_0^h \int_{B'}|\partial_j \textbf{V}(\nabla \textbf{u})|^2\,\left\vert\left(\eta^q f(|\delta_h \textbf{u}|)\right)\circ T_{-he_j}-\eta^q f(|\delta_h \textbf{u}|)\right\vert\ds\\
	&+\dashint_0^h \int_{B'}|\partial_j \textbf{V}(\nabla \textbf{u})|^2\,\left\vert f(|\delta_h \textbf{u}|)- f(v)\right\vert\eta^q\ds\\
	=:&\text{II}''_1+\text{II}_2''+\text{II}_3''+\text{II}_4''
\end{align*}

We have $\text{II}''_1\rightarrow 0$ for the same reasons as $\text{IV}'_{j,1}\rightarrow 0$ and $\text{III}'_{j,1}\rightarrow 0$. The integrands of $\text{II}''_2$ and $\text{II}_3''$ are bounded by the $L^1_{\text{loc}}$-function $\norm{f}_\infty|\partial_j \textbf{V}(\nabla \textbf{u})|^2$ and go to zero for $s\rightarrow 0$ pointwise almost everywhere. This means the integrals over $B'$ go to zero and we can use the fundamental theorem as before. We get $\text{II}_4''\rightarrow 0$ via dominated convergence like $\text{III}''_{j,2}$.\\
This means in the end (using also $\supp \eta\subset B$):

\begin{equation}\label{IIKonvergenz}
	\text{II}_j'\leq \frac{1}{|B|}\text{II}''_j\rightarrow \dashint_B |\partial_j \textbf{V}(\nabla \textbf{u})|^2f(v)\eta^q
\end{equation}

Now we can let $h\rightarrow 0$ in \ref{discreteEnergy} and get using \ref{IKonvergenz}, \ref{IVKonvergenz}, \ref{IIIKonvergenz} and \ref{IIKonvergenz}

\begin{equation}\label{statepsilonenergy}
	\dashint_B |\nabla \textbf{V}(\nabla \textbf{u})|^2\zeta^q f(v)\lesssim 2\epsilon\dashint_B |\nabla \textbf{V}(\nabla \textbf{u})|^2\zeta^q f(v)+c_\epsilon\dashint_B\phi(v)|\nabla\zeta|^2f(v)
\end{equation}
We choose $\epsilon$ small enough that we can absorb the first summand of the right hand side on the left hand side and the proof for $f\in C^1$ is concluded.
\end{proof}

\begin{proof}[Proof of theorem \ref{statenergy}]
For the case of a general non decreasing bounded piecewise differentiable function $f$ approximate it by a sequence of non-decreasing, uniformly bounded $C^1$ functions $f_k$ with $\lim_{k\rightarrow\infty}f_k(x)=f(x)$ for all $x\in\mathbb{R}^+_0$. We use \ref{statenergylemma} and get

$$\dashint_B D_k:=\dashint_B |\nabla \textbf{V}(\nabla \textbf{u})|^2\eta^qf_k(v)\lesssim\dashint_B\phi(v)|\nabla\eta|^2f_k(v)=:\dashint_B E_k$$
As we have $f_k\rightarrow f$ pointwise everywhere, we get $D_k\rightarrow D_\infty$ and $E_k\rightarrow E_\infty$ almost everywhere. As we have $E_k\leq\norm{f}_\infty|\nabla \textbf{V}(\nabla \textbf{u})|^2\eta^q\in L^1(B)$ and $E_k\leq\norm{f}_\infty\phi(v)|\nabla\eta|^2\in L^1(B)$, we can use dominated convergence and get the desired result.
\end{proof}

\begin{corollary}\label{statcor}
 Let $\phi$ be an N-function with $\Delta_2(\{\phi,\phi^\ast\})$ satisfying the assumption \ref{mainassumption} and let $\textbf{u}\in W^{1,\phi}_{\text{loc}}(\Omega,\mathbb{R}^m)$ be a local weak solution to 
 $\Delta_\phi \textbf{u}=0 $ and $G(t):=(\psi'(t)-\psi'(\gamma))_+$ with a non negative real number $\gamma$\\
 Then we have
 \begin{equation}
  \dashint_B |\nabla \left(G(v)\eta^{\frac q 2}\right)|^2\lesssim \dashint_B \phi(v) \chi_{\{v>\gamma\}} |\nabla \eta|^2\\
 \end{equation}
\end{corollary}

\begin{proof}
  We use $f(t)=\chi_{\{t>\gamma\}}$. With theorem \ref{statenergy} we get
  $$\dashint_B |\nabla \textbf{V}(\nabla \textbf{u})|^2\eta^q\chi_{\{v>\gamma\}}\lesssim\dashint_B\phi(v)|\nabla\eta|^2\chi_{\{v>\gamma\}}$$
  For the left hand side we use that $|(|\textbf{Q}|)'|=|\frac {\textbf{Q}} {|\textbf{Q}|}|\leq 1$ and $(x_+)'=\chi_{\mathbb{R}^+}(x)$ which are both bounded which means that we can apply the chain rule for sobolev functions and $\chi_{t>\gamma}=\chi_{\{t>\gamma\}}^2$ almost everywhere:
  \begin{align}
    &\dashint_B |\nabla \textbf{V}(\nabla \textbf{u})|^2\eta^q\chi_{\{v>\gamma\}}\geq\dashint_B |\nabla \left(|\textbf{V}(\nabla \textbf{u})|\right)|^2\eta^q\chi_{\{v>\gamma\}}\nonumber\\
		=&\dashint_B \left\vert\nabla \left(\psi'(v)\right)\chi_{\{v>\gamma\}}\eta^{\frac q 2}\right\vert^2
    \geq\dashint_B \left\vert\nabla \left(\psi'(v)-\psi'(\gamma)\right)\chi_{\{v>\gamma\}}\eta^{\frac q 2}\right\vert^2 \nonumber \\=&\dashint_B \left\vert\nabla \left(\left(\psi'(v)-\psi'(\gamma)\right)_+\right)\eta^{\frac q 2}\right\vert^2 \label{cor11}
  \end{align}
\sloppy  As we also have $G^2(v)\leq\psi'(v)^2\chi_{\{v>\gamma\}}\sim\phi(v)\chi_{\{v>\gamma\}}$ and $|\nabla(\eta^{\frac q 2})|=\frac q 2 \eta^{\frac q 2-1}|\nabla\eta|\lesssim|\nabla\eta|$ we get
  \begin{equation}\label{cor12}
    \dashint_B G^2(v)\left\vert\nabla\left(\eta^{\frac q 2}\right)\right\vert^2\lesssim \dashint_B\phi(v)|\nabla\eta|^2\chi_{\{v>\gamma\}}
  \end{equation}
  After adding \ref{cor11} and \ref{cor12} we conclude the proof with the product rule.
 \end{proof}

\newpage

\subsection{The parabolic case}\label{secInstatEnergy}
 \begin{theorem}[Energy estimate for the inelliptic case] \label{instatenergy}
 Let $\phi$ be an N-function with $\Delta_2(\{\phi,\phi^\ast\})<\infty$ satisfying the assumption \ref{mainassumption} and let $\textbf{u}\in L_{\text{loc}}^\phi(J\times \Omega,\mathbb{R}^m)\cap L_{\text{loc}}^2(J\times \Omega,\mathbb{R}^m)$ with $|\nabla \textbf{u}|:=v\in L^\phi_{\text{loc}}(J\times\Omega)\cap L_{\text{loc}}^2(J\times\Omega)$ be a local weak solution to 
 $$\Delta_\phi \textbf{u}=\partial_t\textbf{u}$$
on a cylindrical domain $J\times\Omega\subset\mathbb{R}^{1+n}$ and let $f:\mathbb{R}^+_0\rightarrow\mathbb{R}$ be a non-decreasing, piecewise continuously differentiable, bounded function which is constant for large arguments. Define $\textbf{V}(\textbf{Q})=\frac{\sqrt{\phi'(|\textbf{Q}|)}}{|\textbf{Q}|}\textbf{Q}$ as usual and $H'(t)=tf(t)$ and  let $Q:=I\times B\Subset J\times\Omega$ be a cylinder where $B$ is a ball in $\mathbb{R}^n$ of radius $R_x$ and $I$ an interval of length $R_t=\alpha R_x^2$ and $\eta$ a $C_0^\infty(Q)$ function with $0\leq\eta\leq 1$.\\
 Then we get
 \begin{align}
  &\sup_I\frac 1 \alpha \dashint_B H(v)\eta^q+R_x^2\dashint_Q |\nabla \textbf{V}(\nabla \textbf{u})|^2\eta^qf(v)\nonumber\\
	\lesssim &R_x^2\dashint_Q|\textbf{V}(\nabla \textbf{u})|^2|\nabla\eta|^2f(v)+R_x^2\dashint_Q H(v)\eta^{q-1}\partial_t\zeta
 \end{align}
\end{theorem}
As in the elliptic case, we start with a lemma restricting $f$ to differentiable functions with $f'\geq 0$.
\begin{lemma}\label{instatenergylemma}
 The assertion of theorem \ref{instatenergy} holds with the additional assumption $f\in C^1$ with $f'(t)=0$ for large $t$.
\end{lemma}

\begin{proof}
	As we do not have (weak) differentiability of $\textbf{u}$ or $v$ in $t$, we need to use a standard mollifier $\xi_\sigma(t)$ in one dimension and denote $g_\sigma=g\ast\xi_\sigma$ This is differentiable in time for all $\sigma>0$ and converges to $g(x,t)$ in $L^\phi(Q)$  for $\sigma\rightarrow 0$ if $g\in L^\phi(Q)$. \\
	For the equation this means using the test function $\textbf{g}$:
	\begin{align*}
		&\int_Q\left[ \textbf{A}(\nabla \textbf{u})\right]_\sigma(t,x)\nabla \textbf{g}(t,x)\dz\\
		=&\int_Q\int  \textbf{A}(\nabla \textbf{u})(t-\tau,x)\xi_\sigma(\tau)\nabla \textbf{g}(t,x)\text{d}\tau \dz\\
		=&\int\int_Q  \textbf{A}(\nabla \textbf{u})(t,x)\nabla \textbf{g}(t+\tau,x)\dz\,\xi_\sigma(\tau)\text{d}\tau \\
		=&\int\int_Q \textbf{u}(t,x)\left(\partial_t \textbf{g}\right)(t+\tau,x)\dz\,\xi_\sigma(\tau)\text{d}\tau \\
		=&\int_Q\int \textbf{u}(t-\tau,x)\,\xi_\sigma(\tau)\text{d}\tau \left(\partial_t \textbf{g}\right)(t,x)\dz\\
		=&\int_Q \textbf{u}_\sigma \left(\partial_t \textbf{g}\right)(t,x)\dz
	\end{align*}
	We now use the test function $g(t,x):=\delta_{h,-j}(f(|\delta_{h}\textbf{u}_\sigma|)\delta_{h,j}\textbf{u}_\sigma\rho(t)\eta^q)$ where $\rho(t)$ is a $C^\infty$-approximation of $\chi_{t>t_0}$ and after a summation over $j$ using Einstein's summation convention and recalling $H'(t)=tf(t)$ we get:
	\begin{align*}
		&\dashint_Q\left[ \textbf{A}(\nabla \textbf{u})\right]_\sigma\nabla \delta_{h,-j}(f(|\delta_{h}\textbf{u}_\sigma|)\delta_{h,j}\textbf{u}_\sigma\rho(t)\eta^q)\dz \displaybreak[0]\\
		=&\dashint_Q \textbf{u}_\sigma \left(\partial_t \delta_{h,-j}(f(|\delta_{h}\textbf{u}_\sigma|)\delta_{h,j}\textbf{u}_\sigma\rho(t)\eta^q)\right)\dz\displaybreak[0]\\
		&\dashint_Q\left[ \delta_{h,j}\textbf{A}(\nabla \textbf{u})\right]_\sigma\nabla (f(|\delta_{h}\textbf{u}_\sigma|)\delta_{h,j}\textbf{u}_\sigma\rho(t)\eta^q)\dz \displaybreak[0]\\
		=&-\dashint_Q \partial_t\delta_{h,j}\textbf{u}_\sigma  f(|\delta_{h}\textbf{u}_\sigma|)\delta_{h,j}\textbf{u}_\sigma\rho(t)\eta^q\dz\displaybreak[0]\\
		=&-\dashint_Q f(|\delta_{h}\textbf{u}_\sigma|)|\delta_{h}\textbf{u}_\sigma|\partial_t|\delta_{h}\textbf{u}_\sigma|  \rho(t)\eta^q\dz\\
		=&-\dashint_Q \partial_t H(|\delta_h \textbf{u}_\sigma|)  \rho(t)\eta^q\dz
		=\dashint_Q  H(|\delta_h \textbf{u}_\sigma|)  \partial_t(\rho(t)\eta^q)\dz\displaybreak[0]\\
		=&\dashint_Q  H(|\delta_h \textbf{u}_\sigma|)  \eta^q\partial_t\rho(t)\dz+\dashint_Q  H(|\delta_h \textbf{u}|_\sigma)  \rho(t)\partial_t\eta^q\dz\\
		=&\dashint_Q  H(|\delta_h \textbf{u}_\sigma|)  \rho(t)\partial_t\eta^q\dz-\dashint_Q \rho(t) \partial_t \left(H(|\delta_h \textbf{u}_\sigma|)  \eta^q\right)\dz\\
		\end{align*}
	We now note that $\chi_{t_0,T}\leq 1$ and let $\rho\rightarrow \chi_{t_0,T}$ (as we have smoothed the functions the limits are easily justified by the dominated convergence theorem) and get
	\begin{align}
		\text{I}+\text{II}:=&\dashint_Q\left[ \delta_{h,j}\textbf{A}(\nabla \textbf{u})\right]_\sigma\nabla (f(|\delta_{h}\textbf{u}_\sigma|)\delta_{h,j}\textbf{u}_\sigma\rho(t)\eta^q)\dz+\frac{1} {R_t}\dashint_B  \left(H(|\delta_h \textbf{u}_\sigma|)  \eta^q\right)\dx\big \vert_{t=T}\nonumber \\
		\leq&\dashint_Q  H(|\delta_h \textbf{u}_\sigma|)\partial_t\left(\eta^q\right)\dz:=\text{III}\label{smoothedversion}
	\end{align}
	We now want to take the limit $\sigma\rightarrow 0$. \\
	\begin{align*}
		\text{I}=&\dashint_Q\left[ \delta_{h,j}\textbf{A}(\nabla \textbf{u})\right]_\sigma (\delta_{h,j}\nabla \textbf{u}_\sigma)f(|\delta_{h}\textbf{u}_\sigma|)\rho(t)\eta^q\dz\\
		+&\dashint_Q\left[ \delta_{h,j}\textbf{A}(\nabla \textbf{u})\right]_\sigma\nabla (f(|\delta_{h}\textbf{u}_\sigma|))\delta_{h,j}\textbf{u}_\sigma\rho(t)\eta^q\dz\\
		+&\dashint_Q\left[ \delta_{h,j}\textbf{A}(\nabla \textbf{u})\right]_\sigma \rho(t)\nabla(\eta^q)\delta_{h,j}\textbf{u}_\sigma f(|\delta_{h}\textbf{u}_\sigma|)\dz=:\text{I}_1+\text{I}_2+\text{I}_3
	\end{align*}
	We note that $\textbf{A}(\nabla \textbf{u})\in L^{\phi^\ast}(Q)$ since
	\begin{equation*}
		\phi^\ast\left(\left\vert\frac{\phi'(v)}{v}\nabla \textbf{u}\right\vert\right)=\phi^\ast(\phi'(v))\sim\phi(v)\in L^1_{\text{loc}}(J\times\Omega)
	\end{equation*}
	And as $L^{\phi^\ast}(Q)$ is a vector space because of $\Delta_2(\phi^\ast)<\infty$, we also have  $\delta_{h,j}\textbf{A}(\nabla \textbf{u})\in L^{\phi^\ast}(Q)$ and therefore $\left[\delta_{h,j}\textbf{A}(\nabla \textbf{u})\right]_\sigma\rightarrow \delta_{h,j}\textbf{A}(\nabla \textbf{u})$ in $L^{\phi^\ast}(Q)$.\\
	This means we have for a general $g\in L^\phi(Q)$ (with therefore $g_\sigma\rightarrow g$ in $L^\phi(Q)$ and $||g_\sigma||_{L^\phi(Q)}$ uniformly bounded):
	
	\begin{align*}
		&\left\vert\dashint_Q \left[\delta_{h,j}\textbf{A}(\nabla \textbf{u})\right]_\sigma g_\sigma-\delta_{h,j}\textbf{A}(\nabla \textbf{u})g \dz\right\vert\\
		\leq&\dashint_Q \left\vert\left[\delta_{h,j}\textbf{A}(\nabla \textbf{u})\right]_\sigma-\delta_{h,j}\textbf{A}(\nabla \textbf{u})\right\vert\,|g_\sigma|\dz+\dashint_Q|\delta_{h,j}\textbf{A}(\nabla u)|\,|g_\sigma-g|\dz\\
		\leq&2\norm{\left[\delta_{h,j}\textbf{A}(\nabla \textbf{u})\right]_\sigma-\delta_{h,j}\textbf{A}(\nabla \textbf{u})}_{L^{\phi^\ast}}\norm{g_\sigma}_{L^\phi}+2\norm{\delta_{h,j}\textbf{A}(\nabla \textbf{u})}_{L^{\phi^\ast}}\norm{g_\sigma-g}_{L^\phi}\rightarrow 0
	\end{align*}
	Using $\delta_{h,j}\nabla \textbf{u}\in L^\phi(Q)$ and dominated convergence we get for $\text{I}_1$:
	\begin{align*}
		&\left\vert\dashint_Q\left[ \delta_{h,j}\textbf{A}(\nabla u)\right]_\sigma (\delta_{h,j}\nabla \textbf{u}_\sigma)f(|\delta_{h}\textbf{u}_\sigma|)\rho(t)\eta^q-\delta_{h,j}\textbf{A}(\nabla \textbf{u}) \delta_{h,j}\nabla \textbf{u} f(|\delta_{h}\textbf{u}|)\rho(t)\eta^q\dz\right\vert\\
		\leq&\norm{f(|\delta_{h}\textbf{u}_\sigma|)\rho(t)\eta^q}_\infty\dashint_Q|\left[ \delta_{h,j}\textbf{A}(\nabla \textbf{u})\right]_\sigma (\delta_{h,j}\nabla \textbf{u}_\sigma)-\delta_{h,j}\textbf{A}(\nabla \textbf{u}) \delta_{h,j}\nabla \textbf{u}|\dz\\
		+&\norm{\rho(t)\eta^q}_\infty\dashint_Q |\delta_{h,j}\textbf{A}(\nabla \textbf{u}) \delta_{h,j}\nabla \textbf{u}|\, |f(|\delta_{h}\textbf{u}_\sigma|)-f(|\delta_{h}\textbf{u}|)|\dz\rightarrow 0
	\end{align*}
	For $\text{I}_2$ we can use the chain rule since $f$ is globally Lipschitz and differentiable:
	\begin{equation*}
		\text{I}_2=\dashint_Q\left[ \delta_{h,j}\textbf{A}(\nabla \textbf{u})\right]_\sigma f'(|\delta_{h}\textbf{u}_\sigma|)\frac{\delta_{h,k}\textbf{u}_\sigma \nabla\delta_{h,k}\textbf{u}_\sigma}{|\delta_{h}\textbf{u}_\sigma|} \delta_{h,j}\textbf{u}_\sigma\rho(t)\eta^q
	\end{equation*}

	We now see that $f'(|\delta_h \textbf{u}_\sigma|)\delta_{h,j}\textbf{u}_\sigma$ is bounded uniformly in $\sigma$ as $f'(t)t$ is bounded and therefore $\norm{f'(|\delta_{h}\textbf{u}_\sigma|)\frac{\delta_{k,h}\textbf{u}_\sigma\delta_{h,j}\textbf{u}_\sigma}{|\delta_{h}\textbf{u}_\sigma|}}_\infty$ is uniformly bounded in $\sigma$. Using this, $\delta_{k,h}\nabla \textbf{u}\in L^\phi(Q)$ and dominated convergece we get
	\begin{align*}
		&\Big\vert\dashint_Q\left[ \delta_{h,j}\textbf{A}(\nabla \textbf{u})\right]_\sigma \delta_{k,h}\nabla \textbf{u}_\sigma f'(|\delta_{h}\textbf{u}_\sigma|)\frac{\delta_{k,h}\textbf{u}_\sigma\delta_{h,j}\textbf{u}_\sigma}{|\delta_{h}\textbf{u}_\sigma|}\rho(t)\eta^q\\
		-& \delta_{h,j}\textbf{A}(\nabla \textbf{u}) \delta_{k,h}\nabla \textbf{u} f'(|\delta_{h}\textbf{u}|)\frac{\delta_{k,h}\textbf{u}\delta_{h,j}\textbf{u}}{|\delta_{h}\textbf{u}|}\rho(t)\eta^q\dz\Big\vert\\
		\leq&\left\vert\left\vert f'(|\delta_{h}\textbf{u}_\sigma|)\frac{\delta_{k,h}\textbf{u}_\sigma\delta_{h,j}\textbf{u}_\sigma}{|\delta_{h}\textbf{u}_\sigma|}\rho(t)\eta^q\right\vert\right\vert_\infty\dashint_Q|\left[ \delta_{h,j}\textbf{A}(\nabla \textbf{u})\right]_\sigma \delta_{k,h}\nabla \textbf{u}_\sigma-\delta_{h,j}\textbf{A}(\nabla \textbf{u}) \delta_{k,h}\nabla \textbf{u}|\dz\\
		&+\dashint_Q \delta_{h,j}\textbf{A}(\nabla u) \delta_{k,h}\nabla \textbf{u} \left\vert f'(|\delta_{h}\textbf{u}_\sigma|)\frac{\delta_{k,h}\textbf{u}_\sigma\delta_{h,j}\textbf{u}_\sigma}{|\delta_{h}\textbf{u}_\sigma|}-f'(|\delta_{h}\textbf{u}|)\frac{\delta_{k,h}\textbf{u}\delta_{h,j}\textbf{u}}{|\delta_{k,h}\textbf{u}|}\right\vert\rho(t)\eta^q\dz\rightarrow 0
	\end{align*}
	Treating $\text{I}_3$ works the same way as treating $\text{I}_1$ using that $\norm{f(|\delta_{h}\textbf{u}_\sigma|)\rho(t)\nabla\zeta}_\infty$ is uniformly bounded in $\sigma$ and $\delta_{h,j}\textbf{u}\in L^\phi(Q)$:
	\begin{align*}
		&\left\vert\dashint_Q\left(\left[ \delta_{h,j}\textbf{A}(\nabla \textbf{u})\right]_\sigma\delta_{h,j}\textbf{u}_\sigma f(|\delta_{h}\textbf{u}_\sigma|)\rho(t)\nabla(\eta^q) -\delta_{h,j}\textbf{A}(\nabla \textbf{u})\delta_{h,j}\textbf{u} f(|\delta_{h}\textbf{u}|)\rho(t)\nabla(\eta^q) \right)\dz\right\vert\\
		\leq&\norm{ f(|\delta_{h}\textbf{u}_\sigma|)\rho(t)\nabla(\eta^q)}_\infty\dashint_Q |\left[ \delta_{h,j}\textbf{A}(\nabla \textbf{u})\right]_\sigma\delta_{h,j}\textbf{u}_\sigma-\delta_{h,j}\textbf{A}(\nabla \textbf{u})\delta_{h,j}\textbf{u}|\dz\\
		&+\dashint_Q |\delta_{h,j}\textbf{A}(\nabla \textbf{u})\delta_{h,j}\textbf{u}| \,| f(|\delta_{h}\textbf{u}_\sigma|)- f(|\delta_{h}\textbf{u}|)|\left\vert\rho(t)\nabla(\eta^q)\right\vert\dz\rightarrow 0
	\end{align*}
	
	We now want to estimate $\text{II}$ and $\text{III}$ in equation \ref{smoothedversion}. For this reason we first note for $b>a$:
	\begin{equation}\label{HAbschaetzung}
		|H(b)-H(a)|=\dashint_a^b s f(s) \ds\leq \norm{f}_\infty\dashint_a^b s\ds =\frac{\norm{f}_\infty}{2}\left(b^2-a^2\right)
	\end{equation}
	and since $\nabla\textbf{u}\in L^2(Q,\mathbb{R}^m)$ we have $|\delta_h \textbf{u}_\sigma|\rightarrow |\delta_h \textbf{u}|$ in $L^2(Q)$ and get taking the limit $\sigma\rightarrow 0$:
	\begin{align*}
		\text{II}&\rightarrow\frac 1 {R_t} \dashint_B  \left(H(|\delta_h \textbf{u}|)  \eta^q\right)\dx\big \vert_{t=T}\\
		\text{III}&\rightarrow\dashint_Q  H(|\delta_h \textbf{u}|)\partial_t\left(\eta^q\right)\dz
	\end{align*}
	This means we can take the limit $\sigma\rightarrow 0$ and the supremum over all $T\in I$ in equation \ref{smoothedversion} and get
	
	\begin{align}
		\text{I'}+\text{II'}:=&\dashint_Q\delta_{h,j}\textbf{A}(\nabla \textbf{u})\nabla (f(|\delta_{h}\textbf{u}_\sigma|)\delta_{h,j}\textbf{u}\rho(t)\eta^q)\dz\nonumber\\
		+& \frac 1 {R_t}\sup_I\dashint_B  \left(H(|\delta_h \textbf{u}|)  \eta^q\right)\dx\leq\dashint_Q  H(|\delta_h \textbf{u}|)\partial_t\left(\eta^q\right)\dz=:\text{III'}\label{instatdiscretenergy}
	\end{align}
	We now want take the limit $h\rightarrow 0$. Since $\textbf{V}(\nabla \textbf{u})\in L_{\text{loc}}^2(J, W_{\text{loc}}^{1,2}(\Omega))$ (see Theorem \ref{instatheorem1}) we can proceed as in the elliptic case (lemma \ref{statenergylemma}) for term $\text{I'}$. For II' and III' we note that $\textbf{u}\in L_{\text{loc}}^2(J,W_{\text{loc}}^{1,2}(\Omega))$ and therefore $|\delta_h \textbf{u}|\rightarrow v$ in $L^2(Q)$ as $h\rightarrow 0$. Using equation \ref{HAbschaetzung} we get
	\begin{align*}
		\text{II'}&\rightarrow \sup_I\frac 1 {R_t}\dashint_B  H(v)  \eta^q\dx\\
		\text{III'}&\rightarrow \dashint_Q  H(v)\partial_t\left(\eta^q\right)\dz
	\end{align*}
	This means we can take the limit $h\rightarrow0$ in equation \ref{instatdiscretenergy} and multiply by $R_x^2$ to get
	\begin{align*}
		&\sup_I\frac 1 {\alpha}\dashint_B H(v)\eta^q+R_x^2\dashint_Q |\nabla \textbf{V}(\nabla \textbf{u})|^2\eta^q f(v)\\
		\lesssim &R_x^2\dashint_Q|\textbf{V}(\nabla \textbf{u})|^2|\nabla\eta|^2f(v)+R_x^2\dashint_Q H(v)\eta^{q-1}\partial_t\eta
	\end{align*}
\end{proof}
\begin{proof}[Proof of theorem \ref{instatenergy}]
	As in the proof of theorem \ref{statenergy}, we approximate $f$ by a sequence of uniformly bounded, non-decreasing $C^1$ functions $f_k$ with $\lim_{k\rightarrow\infty}f_k(x)=f(x)$ for all $x\in\mathbb{R}^+_0$. As the $f_k$ are uniformly bounded $C^1$ functions we can apply lemma \ref{instatenergylemma} and get with $H_k(t):=\int_0^t sf_k(s)\ds$
	\begin{align}
		\sup_I\dashint_B A_k+&\dashint_Q B_k:=\sup_I\frac 1 {\alpha} \dashint_B H_k(v)\eta^q+R_x^2\dashint_Q |\nabla \textbf{V}(\nabla \textbf{u})|^2\eta^qf_k(v)\nonumber\\
		\lesssim&R_x^2\dashint_Q|\textbf{V}(\nabla \textbf{u})|^2|\nabla\eta|^2f_k(v)+R_x^2\dashint_Q H_k(v)\eta^{q-1}\partial_t\eta=:\dashint_Q C_k+\dashint_Q D_k\label{instatenergyepsilon}
	\end{align}
	We have $\norm{f_k}_\infty\leq M$. As in the proof of theorem \ref{statenergy} $B_k$ is bounded by the $L^1$-function $M |\nabla \textbf{V}(\nabla \textbf{u})|^2\eta^q$ and $C_k$ is bounded by $M|\textbf{V}(\nabla \textbf{u})|^2|\nabla\eta|^2 \in L^1(Q)$.\\ 
	For the other terms we note that $H_k(t)=\int_0^t sf_k(s)\ds\leq M s^2$.This means we have $A_k\leq M\,v^2\eta^q\in L^1(Q)$ and $D_k\leq M v^2\eta^{q-1}\partial_t\eta\in L^1(Q)$. This means we can take the limit $k\rightarrow \infty$ and use dominated convergence to conclude the proof.
\end{proof}

\begin{corollary}\label{instatcor}
	Let $\phi$, $\textbf{u}$ and $v$ be as defined above and denote $G(t):=(\psi(t)-\psi(\gamma))_+$ and ${H(t)=(v^2-\gamma^2)_+}$ with a non-negative real number $\gamma$.\\
	Then we get
	\begin{align}
		&\sup_I\frac 1 \alpha \dashint_B H(v)\eta^q+R_x^2\dashint_Q |\nabla \left(G(v)\eta^{\frac q 2}\right)|^2\nonumber\\
		\lesssim &R_x^2\dashint_Q\phi(v)\nabla\eta|^2\chi_{\{v>\gamma\}}+R_x^2\dashint_Q H(v)\eta^{q-1}\partial_t\eta\label{instatcorineq}
	\end{align}
\end{corollary}

\begin{proof}
	We use $f(t)=\chi_{\{t>\gamma\}}$. This leads to $H(t)=\int_\gamma^t s\ds_+=(t^2-\gamma^2)_+$ as claimed. To get $\dashint_Q |\nabla V(\nabla u)|^2\chi_{\{v>\gamma\}}\eta^q\gtrsim\dashint_Q |\nabla \left(G(v)\eta^{\frac q 2}\right)|^2$ we proceed like in the proof of corollary \ref{statcor}. Putting this in the result of theorem \ref{instatenergy} concludes the proof.
\end{proof}

\newpage

\section{De-Giorgi-Techinque}
\subsection{Preliminary Lemmas}
At first we proof two important lemmas.
\begin{lemma}(Fast geometric convergence)\label{converg}
  Let $\alpha>0$,$C>0$ and  $b>1$ be real numbers and $a_k$ a sequence with the properties
  $$ a_{k+1}\leq C b^k a_k^{1+\alpha} $$  
  $$ a_0\leq C^{-\frac{1}{\alpha}}b^{-\frac{1}{\alpha^2}} $$
  Then we have $a_k\leq C^{-\frac{1}{\alpha}} b^{-\frac{1+k\alpha}{\alpha^2}}\rightarrow 0$
\end{lemma}
\begin{proof}
 We use induction:\\
 The base case $k=0$ follows directly from the second property.\\
 The induction step is straightforward: Let $a_k\leq C^{-\frac{1}{\alpha}} b^{-\frac{1+k\alpha}{\alpha^2}}$ for some $k$, then we get
  \begin{align*}
   a_{k+1}&\leq C b^k a_k^{1+\alpha}\leq C b^k \left(C^{-\frac{1}{\alpha}} b^{-\frac{1+k\alpha}{\alpha^2}}\right)^{1+\alpha}\\ 
    &\leq C b^k C^{-1-\frac 1 \alpha} b^{-\frac{1+(k+1)\alpha}{\alpha^2}-k}=C^{-\frac{1}{\alpha}} b^{-\frac{1+(k+1)\alpha}{\alpha^2}}
  \end{align*}

\end{proof}

From this we get an easy
\begin{corollary}\label{cor}
 Let $\alpha>0$, $C>0$, $b>1$ and $\gamma$ be real numbers and $a_k$ a sequence with
 $$a_{k+1}\leq C b^k a_k\left(\frac {a_k}{\gamma}\right)^\alpha $$
 Then we have $a_k\rightarrow 0$ if $\gamma=a_0 C^\frac 1 \alpha b^\frac 1 {\alpha^2}$
\end{corollary}

\begin{proof}
 Use Lemma \ref{converg} on the sequence $\frac{a_k}{\gamma}$.
\end{proof}

\begin{lemma}\label{Dif}
 Let $h\in C^1(\mathbb{R}^+_0)$ be an increasing function with $h(0)=0$, $h(2t)\leq d h(t)$ and $h'(t)\sim \frac{h(t)}{t}$ and let $c\in\mathbb{R}^+$ be a constant and define $c_k=c\left(1-2^{-k}\right)$.\\
 Then we have for $v>c_{k+1}$
 $$h(v)\lesssim 2^{k+1}\left(h(v)-h(c_k)\right)_+ $$
 and the constant only depends on $h$.
\end{lemma}

\begin{proof}
  We calculate:\\
  \begin{align*}
   h(v)&=h(v)-h(c_k)+h(c_k)\\
   &=h(v)-h(c_k)+\frac{h(c_k)}{h(c_{k+1})-h(c_k)}\left(h(c_{k+1})-h(c_k)\right)\\
   &\leq(h(v)-h(c_k))\frac{h(c_{k+1})}{h(c_{k+1})-h(c_k)}\\
   &\leq\frac{h(c_{k+1})}{h(c_{k+1})-h(c_k)}(h(v)-h(c_k))_+
  \end{align*}
  If we have $k=0$, we have $h(c_0)=0$ and the therefore $\frac{h(c_{k+1})}{h(c_{k+1})-h(c_k)}=1$.\\
  For the case $k\geq 1$ we use the intermediate value theorem of differential calculus and for some $t\in(c_k,c_{k+1})$ (implying $\frac c 2 \leq t \leq c$) we get
  \begin{align*}
		\frac{h(c_{k+1})}{h(c_{k+1})-h(c_k)}&=\frac{h(c_{k+1})}{h'(t)\left(c_{k+1}-c_k\right)}\\
		&\sim \frac{h(c_{k+1})t}{h(t)\left(c\left(2^{-k}-2^{-k-1}\right)\right)}\\
		&\lesssim \frac{h(c)}{h\left(\frac c 2\right)}2^{k+1}\\
		&\leq d 2^{k+1}
  \end{align*}
  
\end{proof}
\newpage

\subsection{The elliptic case}\label{secStatBound}
We will start directly with the main theorem of this section
\begin{theorem}\label{statphifinal}
 Let $\phi$ be an N-function with $\Delta_2(\{\phi,\phi^\ast\})<\infty$ which satisfies assumption \ref{mainassumption}, let $\textbf{u}\in W^{1,\phi}_{\text{loc}}(\Omega,\mathbb{R}^m)$ be a local weak solution to $\Delta_\phi \textbf{u} =0$ on a domain $\Omega\subset \mathbb{R}^n$ and $B\subset \Omega$ a ball of radius $R$ with $2B\Subset \Omega$. Furthermore, we denote $v:=|\nabla \textbf{u}|$.\\
 Then we have 
 \begin{equation*}
  \sup_B\phi(v)\lesssim\dashint_{2B}\phi(v)
 \end{equation*}
\end{theorem}

\begin{proof}
	We define
	\begin{align*}
		B_k:&=B(1+2^{-k})\\
		\zeta_k&\in C_0^\infty \text{ with }\\
		\chi_{B_k}&\leq\zeta_k\leq\chi_{B_{k+1}}\\
		|\nabla \zeta_k|&\lesssim \frac{2^k}{R}\\
		\gamma_k:&=\gamma_\infty(1-2^{-k})
	\end{align*}
	where $\gamma_\infty\in\mathbb{R}^+$ is a constant to be chosen later.\\
	In the end we want to use Corollary \ref{cor} on the sequence $W_k:=\norm{\phi(v)\chi_{\{v>\gamma_{k}\}}\zeta_{k}^q}_1$ where $q\geq 2$ is chosen such that $\phi(\zeta_k^{q-1} t)\leq \zeta_k^q\phi(t)$ for all $k\in \mathbb{N}$. We estimate:
	\begin{align*}
		W_{k+1}=\norm{\phi(v)\chi_{\{v>\gamma_{k+1}\}}\zeta_{k+1}^q}_1&\leq\norm{\phi(v)\chi_{\{v>\gamma_{k+1}\}}\zeta_{k+1}^q}_{\frac{n}{n-2}}\norm{\chi_{\{v>\gamma_{k+1}\}}\chi_{\supp \zeta_{k+1}}}_\frac 2 n\\
		&\leq \norm{\phi^\frac 1 2 (v)\chi_{\{v>\gamma_{k+1}\}}\zeta_{k+1}^{\frac q 2}}_{\frac{2n}{n-2}}^2\norm{\chi_{\{v>\gamma_{k+1}\}}\chi_{\supp \zeta_{k+1}}}_\frac 2 n
	\end{align*}
	\sloppy We now observe that with $\psi'(t)=\sqrt{\phi'(t)t}\sim\phi^{\frac 1 2}$ the assumptions of lemma \ref{Dif} are fulfilled because of $\Delta_2(\phi)<\infty$ and we get ${\phi^\frac 1 2(t)}\lesssim {2^{k+1}(\psi'(t)-\psi'(\gamma_k))_+}{=:2^{k+1}G_k(t)}$ like in corollary \ref{statcor} for $v>\gamma_{k+1}$. We use this and Sobolev's inequality where we note that the Sobolev constant is proportional to $R^2$:
	\begin{align*}
		& \norm{\phi^\frac 1 2 (v)\chi_{\{v>\gamma_{k+1}\}}\zeta_{k+1}^{\frac q 2}}_{\frac{2n}{n-2}}^2\norm{\chi_{\{v>\gamma_{k+1}\}}\chi_{\supp \zeta_{k+1}}}_\frac 2 n\\
		\lesssim &2^{2k+2}\norm{G_k(v)\zeta_{k+1}^{\frac q 2}}_{\frac{2n}{n-2}}^2\norm{\chi_{\{v>\gamma_{k+1}\}}\chi_{\supp \zeta_{k+1}}}_\frac 2 n\\
		\lesssim &2^{2k+2}R^2\norm{\nabla\left(G_k(v)\zeta_{k+1}^{\frac q 2}\right)}_2^2\norm{\chi_{\{v>\gamma_{k+1}\}}\chi_{\supp \zeta_{k+1}}}^\frac 2 n_1
	\end{align*}
	Now we can apply corollary \ref{statcor} on the first factor. For the second factor we see that using $\chi_{\{v>\gamma_{k+1}\}}^a=\chi_{\{v>\gamma_{k+1}\}}$ and $\zeta_k\equiv1$ on $\supp \zeta_{k+1}$ we get:
	\begin{align*}
		&\norm{\phi(v)\chi_{\{v>\gamma_{k}\}}\zeta_{k}^q}_a\geq\norm{\phi(v)\chi_{\{v>\gamma_{k+1}\}}\zeta_{k}^q}_a\\
		\geq&\phi(\gamma_{k+1})\norm{\chi_{\{v>\gamma_{k+1}\}}\zeta_{k}^q}_a
		\geq\phi(\gamma_{k+1})\norm{\chi_{\{v>\gamma_{k+1}\}}\chi_{\supp\zeta_{k+1}}}_a
	\end{align*}
	Putting this in our estimate gives
	\begin{align*}
		&2^{2k+2}R^2\norm{\nabla\left(G_k(v)\zeta_{k+1}^{\frac q 2}\right)}_2^2\norm{\chi_{\{v>\gamma_{k+1}\}}\chi_{\supp \zeta_{k+1}}}^\frac 2 n_1\\
		\lesssim &2^{2k+2}R^2\norm{\phi(v)\chi_{v>\gamma_k}|\nabla\zeta_{k+1}|^2}_1\left(\frac{\norm{\phi(v)\chi_{\{v>\gamma_{k}\}}\zeta_{k}^q}_1}{\phi(\gamma_{k+1})}\right)^\frac 2 n
	\end{align*}
	We now observe that $\gamma_{k+1}=\gamma_\infty\left(1-2^{-(k+1)}\right)\geq\frac{\gamma_\infty}{2}$ and therefore $\phi(\gamma_{k+1})\geq\phi\left(\frac{\gamma_\infty}{2}\right)\geq\Delta_2(\phi)\phi(\gamma_\infty)$ and using $|\nabla \zeta|^2\leq 2^{2k}R^{-2}\chi_{\supp\zeta_{k+1}}\leq2^{2k}R^{-2}\zeta_k^q$ we get
	\begin{align*}
		& 2^{2k+2}R^2\norm{\phi(v)\chi_{v>\gamma_k}|\nabla\zeta_{k+1}|^2}_1\left(\frac{\norm{\phi(v)\chi_{\{v>\gamma_{k}\}}\zeta_{k}^q}_1}{\phi(\gamma_{k+1})}\right)^\frac 2 n\\
		\lesssim &2^{4k}\norm{\phi(v)\chi_{v>\gamma_k}\zeta^q_{k}}_1\left(\frac{\norm{\phi(v)\chi_{\{v>\gamma_{k}\}}\zeta_{k}^q}_1}{\phi(\gamma_{\infty})}\right)^\frac 2 n = 2^{4k}W_k\left(\frac{W_k}{\phi\left(\gamma_\infty\right)}\right)^{\frac{2}{n}}
	\end{align*}
	In total we have $W_{k+1}\lesssim2^{4k}W_k\left(\frac{W_k}{\phi\left(\gamma_\infty\right)}\right)^{\frac{2}{n}}$ and can apply corollary \ref{cor} on $W_k$. This means we have $W_k\rightarrow 0$ if $\phi(\gamma_\infty)\sim W_0$ but this gives $\chi_{v>\gamma_\infty}=0$ and therefore $\phi(v)\leq\phi(\gamma_\infty)$ on $\supp\zeta_\infty=B$. So in the end we get on $B$:
	\begin{equation*}
		\phi(v)<\phi\left(\gamma_\infty\right)\sim a_0=\dashint_{2B}\phi(v)\chi_{v>0}\zeta_0^2\leq\dashint_{2B}\phi(v)
	\end{equation*}
\end{proof}
\newpage
\subsection{The parabolic case}\label{secInstatBound}
At first we define for a sequence of $C_0^\infty$-functions $\zeta_k$ the norm
$$\norm{f}_{L^s\left(L^r\right)(k)}:=\norm{\norm{f}_{L^s\left(\zeta_k^q\dx \right)}}_{L^r(\dt)}=\left(\dashint\left(\dashint f^r\zeta_k^q\dx\right)^{\frac s r}\dt\right)^{\frac 1 s} $$
and based on this 
\begin{align*}
	Y_k&:=\norm{\phi(v)\chi_{\{v>\gamma_{k}\}}}_{L^1\left(L^1\right)(k)} \\
	Z_k&:=\norm{v^2\chi_{\{v>\gamma_{k}\}}}_{L^1\left(L^1\right)(k)} \\
	W_k&:= Y_k+\frac 1 \alpha Z_k
\end{align*}

\begin{lemma}\label{energy2}
	Let $\textbf{u}\in L^\phi_{\text{loc}}(J\times\Omega,\mathbb{R}^m)\cap C_\text{loc}(I,L^2_{\text{loc}}(\Omega,\mathbb{R}^m))$ with $v:=|\nabla \textbf{u}|\in L^\phi_{\text{loc}}(J\times\Omega)\cap L^2_{\text{loc}}(\Omega)$ be a local weak solution to $\partial_t\textbf{u}-\Delta_\phi \textbf{u}=0$ on a cylindrical domain $J\times\Omega\subset\mathbb{R}^{1+n}$ and let $Q=I\times B \subset \mathbb{R}^{1+n}$ be a cylinder in space-time with Radius $R_x$ in space and height $R_t$ in time with $R_t=\alpha R_x^2$. Let the sequences $\zeta_k\in C_0^\infty\left(\mathbb{R}^{1+n}\right)$ and $\gamma_k\in\mathbb{R}^+$ have the following properties:
	\begin{align*}
		Q_k&=2\left(1+2^{-k}\right)Q=:I_k\times B_k\\
		\chi_{Q_k}&\leq\zeta_k\leq\chi_{Q_{k+1}}\\
		\left\vert\nabla \left(\zeta^{\frac{n-2}{n}}_k\right)\right\vert&\lesssim R_x^{-1}2^k\\
		\left\vert\partial_t \left(\zeta^{\frac{n-2}{n}}_k\right)\right\vert&\lesssim R_t^{-1}2^k\\
		\gamma_k&=\gamma_\infty\left(1-2^{-k}\right)
	\end{align*}
	Then we have
	\begin{align}
		\norm{v^2\chi_{\{v>\gamma_{k+1}\}}}_{L^\infty\left(L^1\right)(k+1)}&\lesssim 2^{3k} \alpha W_k \label{instatabsch1}\\
		\norm{\phi(v)\chi_{\{v>\gamma_{k+1}\}}}_{L^1\left(L^{\frac{n}{n-2}}\right)(k+1)}&\lesssim 2^{3k} W_k\label{instatabsch2}
	\end{align}
\end{lemma}

\begin{proof}
	We recall the energy inequality \ref{instatcorineq} from corollary \ref{instatcor} with $\eta=\left(\zeta^{\frac{n-2}{n}}_{k+1}\right)$:
	\begin{align}
		&\sup_I\frac 1 \alpha \dashint_B H_k(v)\zeta_{k+1}^{q\frac{n}{n-2}}\dx+R_x^2\dashint_Q\left\vert\nabla\left(G\zeta_{k+1}^{\frac q 2 \frac{n}{n-2}}\right)\right\vert^2\dz\nonumber\\
		\lesssim&R_x^2\dashint_Q \phi(v)\left\vert\nabla\left(\zeta_{k+1}^{\frac{n}{n-2}}\right)\right\vert^2\chi_{\{v>\gamma_{k+1}\}}\dz+R_x^2\dashint_Q H(v)\zeta_{k+1}^{(q-1)\frac{n}{n-2}}\partial_t\left(\zeta^{\frac{n}{n-2}}\right)\dz\label{energyzeta}
	\end{align}
	
	At first we estimate the terms on the right hand side of \ref{energyzeta} and note that $\zeta_k\equiv 1$ on $\supp \zeta_{k+1}$:
	
	\begin{align*}
		R_x^2\dashint_Q \phi(v) \chi_{v>\gamma_k}\left\vert\nabla\left(\zeta_{k+1}^{\frac{n-2}{n}}\right)\right\vert^2\dz&\lesssim2^{2k}\dashint_Q \phi(v) \chi_{v>\gamma_k}\chi_{\supp\chi_{k+1}}\dz\\
		&\leq2^{2k}\dashint_Q \phi(v) \chi_{v>\gamma_k}\zeta_k^q\dz\\
		&=2^{2k}Y_k
	\end{align*}

	\begin{align*}
		R_x^2\dashint_Q H_k(v)\left(\zeta^{\frac{n-2}{n}}_{k+1}\right)^{q-1}\left\vert\partial_t\left(\zeta^{\frac{n-2}{n}}_{k+1}\right)\right\vert\dz&\lesssim\frac{2^{k+1}R_x^2}{R_t}\dashint_Q v^2 \chi_{v>\gamma_k}\chi_{\supp\chi_{k+1}}\dz \\
		&\lesssim\frac{2^{k}}{\alpha}\dashint_Q v^2 \chi_{v>\gamma_k}\zeta_k^q\dz\\
		&=\frac{2^{k}}{\alpha}Z_k\leq \frac{2^{2k}}{\alpha}Z_k
	\end{align*}	
	Putting this in \ref{energyzeta} gives
	\begin{equation}\label{instatenergyright}
		\sup_I\frac 1 \alpha \dashint_B H_k(v)\zeta_{k+1}^{q\frac{n}{n-2}}\dx+R_x^2\dashint_Q\left\vert\nabla\left(G\zeta_{k+1}^{\frac q 2 \frac{n}{n-2}}\right)\right\vert^2\dz\lesssim 2^{2k}W_k
	\end{equation}

	To prove \ref{instatabsch1} we use lemma \ref{Dif} with $h(t)=t^2$ to get $v^2\lesssim 2^k H_k(v)$ for $v>\gamma_{k+1}$ and we see that $\zeta\leq\zeta^{\frac {n-2} n}$ as $0\leq\zeta\leq 1$. Putting this in \ref{instatenergyright} gives
	\begin{align*}
		\norm{v^2\chi_{\{v>\gamma_{k+1}\}}}_{L^\infty\left(L^1\right)(k+1)}&=\alpha\sup_{I}\frac{1}{\alpha}\dashint_B v^2\chi_{\{v>\gamma_{k+1}\}}\zeta_{k+1}^q\dx\\
		&\lesssim \alpha2^k\sup_{I}\frac{1}{\alpha}\dashint_B H_k(v)\left(\zeta^{\frac{n-2}{n}}_{k+1}\right)^q\dx\\
		&\lesssim \alpha2^{3k}W_k
	\end{align*}
	
	For inequality \ref{instatabsch2} we set $h(t)=\phi(t)^{\frac{1}{2}}$ in lemma \ref{Dif} and get $\phi(t)^{\frac 1 2}\lesssim 2^k G_k(t)$ for $t>\gamma_{k+1}$ like in the elliptic case. We also use Sobolev's inequality and note that its constant is proportional to $R_x^2$.
	\begin{align*}
		\norm{\phi(v)\chi_{\{v>\gamma_{k+1}\}}}_{L^1\left(L^{\frac{n}{n-2}}\right)(k+1)}&=\norm{\norm{\phi(v)\chi_{\{v>\gamma_{k+1}\}}\zeta_{k+1}^{q\frac{n-2}{n}}}_{L^{\frac{n}{n-2}}(\dx)}}_{L^1(\dt)}\\
		&=\norm{\norm{\phi(v)^{\frac 1 2}\chi_{\{v>\gamma_{k+1}\}}\zeta_{k+1}^{\frac q 2\frac{n-2}{n}}}_{L^{\frac{2n}{n-2}}(\dx)}^2}_{L^1(\dt)}\\
		&\lesssim 2^k\norm{\norm{G_k(v)\zeta_{k+1}^{\frac q 2 \frac{n-2}{n}}}_{L^{\frac{2n}{n-2}}(\dx)}^2}_{L^1(\dt)}\\
		&\lesssim 2^kR_x^2\norm{\norm{\nabla\left(G_k(v)\zeta^{\frac q 2\frac{n-2}{n}}\right)}_{L^2(\dx)}^2}_{L^1(\dt)}\\
		&=2^k R_x^2\dashint \left\vert\nabla\left(G_k(v)\zeta^{\frac q 2\frac{n-2}{n}}\right)\right\vert^2\dz\\
		&\lesssim 2^{3k}W_k
	\end{align*}
\end{proof}
We will now specialize to the case $\phi(t)=t^p$. To find the optimal upper bound in the parabolic $p$-Laplacian case we want to use all the information we get from the lemma we have just proved. With the weak type estimate
\begin{equation}\label{weaktype}
	\norm{v\chi_{\{v>\gamma_{k+1}\}}}_{L^r(L^q)(k+1)}>\gamma_{k+1}\norm{\chi_{\{v>\gamma_{k+1}\}}}_{L^r(L^q)(k+1)}
\end{equation}
we get
\begin{equation}\label{pinfo}
	\begin{aligned}
	\norm{v\chi_{\{v>\gamma_{k+1}\}}}_{L^p\left(L^{p\frac{n}{n-2}}\right)(k+1)}&\lesssim 2^{\frac{3k}{p}}W_k^{\frac 1 p}\\
	\norm{v\chi_{\{v>\gamma_{k+1}\}}}_{L^\infty\left(L^2\right)(k+1)}&\lesssim \alpha^{\frac 1 2} 2^{\frac{3k}{2}}W_k^{\frac 1 2}\\
	\norm{\chi_{\{v>\gamma_{k+1}\}}}_{L^p\left(L^p\right)(k+1)}\leq\ \frac{\norm{v\chi_{\{v>\gamma_{k+1}\}}}_{L^p\left(L^p\right)(k+1)}}{\gamma_{k+1}}&\lesssim 2^{\frac{3k}{p}}\frac{W_k^{\frac 1 p}}{\gamma_\infty}\\
	\norm{\chi_{\{v>\gamma_{k+1}\}}}_{L^\infty\left(L^2\right)(k+1)}\leq\ \frac{\norm{v\chi_{\{v>\gamma_{k+1}\}}}_{L^\infty\left(L^2\right)(k+1)}}{\gamma_{k+1}}&\lesssim 2^{\frac{3k}{2}}\frac{\alpha^{\frac 1 2} W_k^{\frac 1 2}}{\gamma_\infty}\\
	\norm{\chi_{\{v>\gamma_{k+1}\}}}_{L^\infty\left(L^2\right)(k+1)}&\leq 1
	\end{aligned}
\end{equation}
As in the elliptic case we want to apply corollary \ref{cor} on $W_k$. To get to the point where this is possible we use at first H\"older's inequality and then use the interpolation of Bochner-Lebesgue-spaces (cf lemma \ref{interpol} in the appendix) in both factors between the spaces where we have information about the norms.\\

We start by estimating $Y$. For simplicity we drop the $2^k$-factors for now.
\begin{align*}
	Y_{k+1}^{\frac 1p}=&\norm{v\chi_{\{v>\gamma_{k+1}\}}}_{L^p(L^p)(k+1)}\\
	\leq& \norm{v\chi_{\{v>\gamma_{k+1}\}}}_{L^r(L^s)(k+1)}\norm{\chi_{\{v>\gamma_{k+1}\}}}_{L^{r'}(L^{s'})(k+1)}\nonumber\\
	\leq &\norm{v\chi_{\{v>\gamma_{k+1}\}}}_{L^p\left(L^{p \frac{n}{n-2}}\right)(k+1)}^\theta \norm{v\chi_{\{v>\gamma_{k+1}\}}}_{L^\infty(L^2)(k+1)}^{1-\theta}\nonumber\\
	&\norm{\chi_{\{v>\gamma_{k+1}\}}}_{L^p\left(L^{p \frac{n}{n-2}}\right)(k+1)}^{\alpha_1}\norm{\chi_{\{v>\gamma_{k+1}\}}}_{L^\infty(L^2)(k+1)}^{\alpha_2}\norm{\chi_{\{v>\gamma_{k+1}\}}}_{L^\infty(L^\infty)(k+1)}^{\alpha_3}\nonumber\\
	\lesssim& \frac{W_k^{\frac{\theta}{p}+\frac{1-\theta}{2}+\frac{\alpha_1}{p}+\frac{\alpha_2}{2}}\alpha^{\frac{1-\theta+\alpha_2}{2}}}{\gamma_\infty^{\alpha_1+\alpha_2}}\nonumber
\end{align*}
This can be rearranged to
	\begin{equation}
	Y_{k+1}\lesssim W_k \left(\frac{W_k \alpha^{\frac p 2 \frac{1-\theta+\alpha_2}{p(\frac{\theta}{p}+\frac{1-\theta}{2}+\frac{\alpha_1}{p}+\frac{\alpha_2}{2})-1}}}{\gamma_\infty^\frac{p(\alpha_1+\alpha_2)}{p(\frac{\theta}{p}+\frac{1-\theta}{2}+\frac{\alpha_1}{p}+\frac{\alpha_2}{2})-1}}\right)^{p(\frac{\theta}{p}+\frac{1-\theta}{2}+\frac{\alpha_1}{p}+\frac{\alpha_2}{2})-1}\label{YAbsch}
\end{equation}
To fix the parameters we get the equations
\begin{equation}\label{YBed}
	\begin{aligned}
		\frac 1 p &=\frac{1}{r}+\frac{1}{r'}=\frac{1}{s}+\frac{1}{s'}\\
		\frac{1}{r}&=\frac{\theta}{p}\\
		\frac{1}{s}&=\frac{\theta}{p\frac{n}{n-2}}+\frac{1-\theta}{2}\\
		\frac{1}{r'}&=\frac{\alpha_1}{p}\\
		\frac{1}{s'}&=\frac{\alpha_1}{p\frac{n}{n-2}}+\frac{\alpha_2}{2}\\
		1&=\alpha_1+\alpha_2+\alpha_3\\
	\end{aligned}
\end{equation}
From this we get

\begin{align}
	\alpha_1&=1-\theta\nonumber\\
	\alpha_2&=\frac{n p(\theta-1)+4}{n p}\label{Yalpha} \\
	\alpha_3&=\frac{np-4}{pn}\nonumber
\end{align} 
and we are free to choose $\theta\in (0,1)$ as long as we ensure that the $\alpha_i$ are non-negative. For $\alpha_1$ this is always the case. To get $\alpha_2\geq0$ we just have to choose $\theta$ large enough to have $\frac{np-4}{np}<\theta$. As $\alpha_3$ is not dependent on $\theta$, we have to deal with the restriction $np\geq4$ in another way. This will be done later. For now we just note that because of $n\geq2$, we do not have problems for $p\geq2$. We put \ref{Yalpha} in \ref{YAbsch} and get:
\begin{equation}\label{YAbschfinal}
	Y_{k+1}\lesssim W_k \left(\frac{W_k\alpha}{\gamma_\infty^2}\right)^\frac 2n
\end{equation}

We will now do the same for $Z$:
\begin{align*}
  Z_{k+1}^{\frac 1 2}=&\norm{v\chi_{\{v>\gamma_{k+1}\}}}_{L^2(L^2)(k+1)}\leq \norm{v\chi_{\{v>\gamma_{k+1}\}}}_{L^r(L^s)(k+1)}\norm{\chi_{\{v>\gamma_{k+1}\}}}_{L^{r'}(L^{s'})(k+1)}\\
	\leq &\norm{v\chi_{\{v>\gamma_{k+1}\}}}_{L^p(L^{p \frac{n}{n-2}})(k+1)}^\theta \norm{v\chi_{\{v>\gamma_{k+1}\}}}_{L^\infty(L^2)(k+1)}^{1-\theta}\\
	&\norm{\chi_{\{v>\gamma_{k+1}\}}}_{L^p(L^{p \frac{n}{n-2}})(k+1)}^{\alpha_1}\norm{\chi_{\{v>\gamma_{k+1}\}}}_{L^\infty(L^2)(k+1)}^{\alpha_2}\norm{\chi_{\{v>\gamma_{k+1}\}}}_{L^\infty(L^\infty)(k+1)}^{\alpha_3}\\
    \leq& \frac{W_k^{\frac{\theta}{p}+\frac{1-\theta}{2}+\frac{\alpha_1}{p}+\frac{\alpha_2}{2}}\alpha^{\frac{1-\theta+\alpha_2}{2}}}{\gamma_\infty^{\alpha_1+\alpha_2}}
\end{align*}
This can be rearranged to
\begin{equation}\label{ZAbsch}
	Z_{k+1}\lesssim W_k \left(\frac{W_k \alpha^{\frac{1-\Theta+\alpha_2}{2\left(\frac{\theta}{p}+\frac{1-\theta}{2}+\frac{\alpha_1}{p}+\frac{\alpha_2}{2}\right)-1}}}{\gamma_\infty^{\frac{2(\alpha_1+\alpha_2)}{{2\left(\frac{\theta}{p}+\frac{1-\theta}{2}+\frac{\alpha_1}{p}+\frac{\alpha_2}{2}\right)-1}}}}\right)^{2\left(\frac{\theta}{p}+\frac{1-\theta}{2}+\frac{\alpha_1}{p}+\frac{\alpha_2}{2}\right)-1}
\end{equation}
We can substitute $p$ by $2$ in the first equation of \ref{YBed} and get
\begin{equation}\label{ZAbschalpha}
	\begin{aligned}
		\alpha_1&=\frac 1 2 p-\theta\\
		\alpha_2&=\frac{n (\theta-1)+2}{n}\\
		\alpha_3&=\frac{n(4-p)-4}{2n}
	\end{aligned}
\end{equation}
One more time we are allowed to choose $\Theta$ freely between $0$ and $1$ if we ensure that the $\alpha_i$ are non-negative. For this to be possible for $\alpha_1$ and $\alpha_2$ we need a $\Theta$ with $\frac 1 2 p\geq\Theta\geq 1-\frac 2 n$. This is only possible for $p\geq 2-\frac 4 n$. $\alpha_3$ is independent of $\Theta$ and we need $n(4-p)-4\geq 0$ which means $p\leq 2(2-\frac 2 n)$.\\
In this case we put \ref{ZAbschalpha} in \ref{ZAbsch} and get using $\nu_2:=\frac n 2 (p-2) +4$:
\begin{equation}\label{ZAbschfinal}
 Z_{k+1}\lesssim W_k\left(\frac{W_k\alpha}{\gamma_\infty^{\frac{\nu_2}{2}}}\right)^\frac 2 n 
\end{equation}

To rule out most of the restrictions on $p$ we first note that for $n\geq 2$ the requirement $p\leq 2(2-\frac 2 n)$ can only be a problem for $p\geq 2$. We recall that we did not have problems in this case with our estimate of $Y$. So we set $\frac 1 2 =\frac 1 p+\frac 1 q$ and use H\"older, $\chi_{\{v>\gamma_{k+1}\}}(x)\in\{0,1\}$, the weak type estimate \ref{weaktype} and \ref{YAbschfinal}:

\begin{align*}
		Z_{k+1}=&\norm{v\chi_{\{v>\gamma_{k+1}\}}}_{L^2\left(L^2\right)(k+1)}^2\leq\norm{v\chi_{\{v>\gamma_{k+1}\}}}_{L^p\left(L^p\right)(k+1)}^2\norm{\chi_{\{v>\gamma_{k+1}\}}}_{L^q\left(L^q\right)(k+1)}^2\\
		=&\norm{v\chi_{\{v>\gamma_{k+1}\}}}_{L^p\left(L^p\right)(k+1)}^2\norm{\chi_{\{v>\gamma_{k+1}\}}}_{L^p\left(L^p\right)(k+1)}^\frac {2p} {q}\lesssim\frac{\norm{v\chi_{\{v>\gamma_{k+1}\}}}_{L^p\left(L^p\right)(k+1)}^p}{\gamma_\infty^{\frac{2p}{q}}}\\
		=&\frac{Y_{k+1}}{\gamma_\infty^{p-2}}\leq W_k\left(\frac{W_k\alpha}{\gamma^\frac{\nu_2}{2}}\right)^\frac 2n
\end{align*}
This shows that \ref{ZAbschfinal} is true for all $p\geq 2-\frac 4 n$.\\
In an analogous way we are now also able to get rid of the restriction $np\geq4$ in the estimate of $Y$ as we see that this is only a problem for $p\leq 2$. We set $\frac 1 p =\frac 1 2+\frac 1 q$ and use H\"older's inequality, $\chi_{\{v>\gamma_{k+1}\}}(x)\in\{0,1\}$, the weak type estimate \ref{weaktype} and \ref{ZAbschfinal} to get
\begin{align*}
		Y_{k+1}=&\norm{v\chi_{\{v>\gamma_{k+1}\}}}_{L^p\left(L^p\right)(k+1)}^p\leq\norm{v\chi_{\{v>\gamma_{k+1}\}}}_{L^2\left(L^2\right)(k+1)}^p\norm{\chi_{\{v>\gamma_{k+1}\}}}_{L^q\left(L^q\right)(k+1)}^p\\
		=&\norm{v\chi_{\{v>\gamma_{k+1}\}}}_{L^2\left(L^2\right)(k+1)}^p\norm{\chi_{\{v>\gamma_{k+1}\}}}_{L^2\left(L^2\right)(k+1)}^\frac {2p} {q}\lesssim\frac{\norm{v\chi_{\{v>\gamma_{k+1}\}}}_{L^2\left(L^2\right)(k+1)}^2}{\gamma_\infty^{\frac{2p}{q}}}\\
		=&\frac{Y_{k+1}}{\gamma_\infty^{2-p}}\leq W_k\left(\frac{W_k\alpha}{\gamma^2}\right)^\frac 2n
\end{align*}
This means \ref{YAbschfinal} is valid for all $p>1$. If we now add \ref{YAbschfinal} and $\frac{1}{\alpha}$ times \ref{ZAbschfinal} we get the estimate for $W$:

\begin{equation}
	W_{k+1}\lesssim W_k\left(\min\left\{\frac{W_k\alpha}{\gamma_\infty^2},\frac{W_k\alpha^{\frac{2-n}{n}}}{\gamma_\infty^{\frac{\nu_2}{2}}}\right\}\right)^{\frac{2}{n}}
\end{equation}
We see that this is independent of $\Theta$ and we still have the assumption $p>2-\frac 4 n$. Assuming this a priori leads to an easier proof of those estimates (and therefore estimates on $v$ via corollary \ref{cor}).
\begin{theorem}
	Let $p>2-\frac 4 n$ and $\textbf{u}\in L^p_{loc}(J,W^{1,p}_{loc}(\Omega,\mathbb{R}^m))\cap C_{\text{loc}}(J,L^2_{loc}(\Omega,\mathbb{R}^m))$ with $v:=|\nabla \textbf{u}|\in L^2_{\text{loc}}(J\times\Omega)$ be a local weak solution to the parabolic $p$-Laplacian equation $\partial_t \textbf{u}-\Delta_p \textbf{u}=0$ on a cylindrical Domain $J\times\Omega\subset \mathbb{R}^{1+n}$. Denote $\nu_2=n(p-2)+4$. For a cylinder $Q=I\times B$ with $2Q\Subset J\times\Omega$ and $R_t=\alpha R_x^2$ as before we have
	$$\min\left\{\frac {v^{\frac {\nu_2} 2}}{\alpha^{\frac{2-n}{n}}},\frac{v^2}{\alpha}\right\}\leq \dashint_{2Q}\frac{v^2}{\alpha}+v^p $$
\end{theorem}

\begin{proof}
	We use the definitions from lemma \ref{energy2} and get using equations \ref{pinfo} and \ref{weaktype}:
	\begin{align*}
		Y_{k+1}=&\norm{v\chi_{\{v>\gamma_{k+1}\}}}_{L^p\left(L^p\right)(k+1)}^p\\
		\leq&\norm{v\chi_{\{v>\gamma_{k+1}\}}}_{L^p\left(L^{p\frac{n}{n-2}}\right)(k+1)}^p\norm{\chi_{\{v>\gamma_{k+1}\}}}_{L^\infty\left(L^{\frac{pn}{2}}\right)(k+1)}^p\\
		=&\norm{v\chi_{\{v>\gamma_{k+1}\}}}_{L^p\left(L^{p\frac{n}{n-2}}\right)(k+1)}^p\norm{\chi_{\{v>\gamma_{k+1}\}}}_{L^\infty\left(L^2\right)(k+1)}^{\frac 4 n}\\
		=&\norm{v\chi_{v<\gamma_{k+1}}}_{L^p\left(L^p\right)(k+1)}^p\norm{\chi_{v<\gamma_{k+1}}}_{L^\infty\left(L^2\right)(k+1)}^{\frac{4}{n}}\\
		\lesssim& 2^{3k}W_k \frac {2^{3k\frac{2}{n}} W_k^{\frac 2 n}\alpha^{\frac{2}{n}}}{\gamma_\infty^\frac 4 n}=2^{3k\left(1+\frac 2 n\right)}W_k\left(\frac{W_k\alpha}{\gamma^2}\right)^{\frac 2 n}
	\end{align*}

	To estimate $Z$ note that for $p>2-\frac 4 n$ the function $t^{p-2+\frac 4 n}$ is increasing.
	\begin{align*}
		Z_{k+1}=&\norm{v\chi_{\{v>\gamma_{k+1}\}}}_{L^2\left(L^2\right)(k+1)}^2=\norm{v^2\chi_{\{v>\gamma_{k+1}\}}}_{L^1\left(L^1\right)(k+1)}\\
		=&\left\vert\left\vert\frac{v^{p-2+\frac 4 n}}{v^{p-2+\frac 4 n}}v^2\chi_{\{v>\gamma_{k+1}\}}\right\vert\right\vert_{L^1\left(L^1\right)(k+1)}\\
		\leq&\frac{1}{\gamma_{k+1}^{p-2+\frac 4 n}}\norm{{v^{p+\frac 4 n}}\chi_{\{v>\gamma_{k+1}\}}}_{L^1\left(L^1\right)(k+1)}\\
		\lesssim&\frac 1 {\gamma_\infty^{\frac{2\nu_2}{n}}}\norm{v^p\chi_{\{v>\gamma_{k+1}\}}}_{L^1\left(L^{\frac n {n-2}}\right)(k+1)}\norm{v^{\frac 4 n}\chi_{\{v>\gamma_{k+1}\}}}_{L^\infty\left(L^{\frac 2 n}\right)(k+1)}\\
		=&\frac 1 {\gamma^{\frac{\nu_2}{n}}}\norm{v\chi_{\{v>\gamma_{k+1}\}}}^p_{L^p\left(L^{p\frac n {n-2}}\right)(k+1)}\norm{v\chi_{\{v>\gamma_{k+1}\}}}_{L^\infty\left(L^{2}\right)(k+1)}^{\frac 4 n}\\
		\lesssim&2^{3k\left(1+\frac 2 n\right)}W_k\left(\frac{\alpha W_k}{\gamma^{\frac{\nu_2} 2}}\right)^{\frac 2 n}
	\end{align*}
	This means we have

	\begin{align*}
		W_{k+1}=&Y_{k+1}+\frac 1 \alpha Z_{k+1}\\
		\lesssim &2^{3k\left(1+\frac 2 n\right)}W_k\left(\alpha\frac{W_k}{\gamma_\infty^2}\right)^{\frac 2 n}+2^{3k\left(1+\frac 2 n\right)}\frac 1 \alpha W_k\left(\frac{W_k\alpha}{\gamma_\infty^{\frac{\nu_2} 2}}\right)^{\frac 2 n}\\
		\lesssim &2^{3k\left(1+\frac 2 n\right)}W_k \max\left\{\left(\frac{W_k\alpha}{\gamma_\infty^2}\right)^{\frac 2 n},\left(\frac{W_k\alpha^{\frac {2-n}{2}}}{\gamma_\infty^{\frac{\nu_2} 2}}\right)^{\frac 2 n}\right\}\\
		=&2^{3k\left(1+\frac 2 n\right)}W_k\left(\frac{W_k}{\min\left\{\frac{\gamma_\infty^{\frac{\nu_2} 2}}{\alpha^{\frac{2-n}{2}}},\frac{\gamma_\infty^2}{\alpha}\right\}}\right)^{\frac 2 n}
	\end{align*}

Like in the elliptic case we conclude with corollary \ref{cor} that $W_k\rightarrow 0$ for $W_0\sim \min\left\{\gamma^{\frac{\nu_2} 2},\gamma^2\right\}$ and we therefore get on $Q$: 
$$  \min\left\{\frac{v^{\frac{\nu_2} 2}}{\alpha^{\frac{2-n}{2}}},\frac{v^2}{\alpha}\right\}<\min\left\{\frac{\gamma_\infty^{\frac{\nu_2} 2}}{\alpha^{\frac{2-n}{2}}},\frac{\gamma_\infty^2}{\alpha}\right\}\sim W_0=\dashint v^p+\frac{v^2}{\alpha}$$
\end{proof}

We remark that we have $\frac{\nu_2}{2}<p$ for $p<2$ and $\frac{\nu_2}{2}>p$ for $p>2$.\\

To generalize the $p$-Laplacian case back to the $\phi$-Laplacian we have to ``translate'' the assumptions on $p$ to assumptions on an N-function $\phi$. As we do not have an easy relationship between $\norm{f}_\phi=\inf\left\{k>0:\,\int \frac {\phi} k\leq 1 \right\}$ and $\int{\phi(v)}$ we cannot use Bochner spaces like before. The proof we got at the end of the previous section is nonetheless easy to generalize. The final theorem of this thesis reads:
\begin{theorem}
	Let $\phi$ be an N-Function with $\Delta_2(\{\phi,\phi^\ast\})<\infty$ satisfying assumption \ref{mainassumption} where $\rho(t)^{\frac 2 n}:=\phi(t)t^{\frac 4 n -2}$ is an increasing function and let $\textbf{u}\in L^\phi_\loc(J\times\Omega)\cap C_{\text{loc}}(J,L^2_{loc}(\Omega,\mathbb{R}^m))$ with $v:=|\nabla \textbf{u}|\in L^\phi_{\text{loc}}(J\times\Omega)\cap L^2_{\text{loc}}(J,L^2_{\text{loc}}(\Omega))$ be a local weak solution to the parabolic $\phi$-Laplacian equation 
	$$\partial_t \textbf{u}-\Delta_\phi \textbf{u}=0$$ 
	on a cylindrical domain $J\times\Omega $. For a cylinder $Q=I\times B$ with $2Q\Subset J\times\Omega$ and $R_t=\alpha R_x^2$ we have
	$$\min\left\{\frac {\rho(v)}{\alpha^{\frac{2-n}{2}}},\frac{v^2}{\alpha}\right\}\lesssim \dashint_{2Q}\frac{v^2}{\alpha}+\phi(v) $$
\end{theorem}

\begin{proof}
	We proceed as we did in the $p$-Laplacian case and use the definitions from lemma \ref{energy2}. For $Y$ we get:
	\begin{align*}
		Y_{k+1}=&\norm{\phi(v)\chi_{\{v>\gamma_{k+1}\}}}_{L^1\left(L^1\right)(k+1)}\\
		\leq&\norm{\phi(v)\chi_{\{v>\gamma_{k+1}\}}}_{L^1\left(L^{\frac{n}{n-2}}\right)(k+1)}\norm{\chi_{\{v>\gamma_{k+1}\}}}_{L^\infty\left(L^{\frac{n}{2}}\right)(k+1)}\\
		=&\norm{\phi(v)\chi_{\{v>\gamma_{k+1}\}}}_{L^1\left(L^{\frac{n}{n-2}}\right)(k+1)}\norm{\chi_{\{v>\gamma_{k+1}\}}}_{L^\infty\left(L^2\right)(k+1)}^{\frac 4 n}\\
		\lesssim& 2^{3k\left(1+\frac 2 n\right)}W_k\left(\frac{W_k\alpha}{\gamma_\infty^2}\right)^{\frac 2 n}
	\end{align*}
	And now for $Z$:
	\begin{align*}
		Z_{k+1}=&\norm{v^2\chi_{\{v>\gamma_{k+1}\}}}_{L^1\left(L^1\right)(k+1)}=\norm{\frac{\rho(v)^{\frac 2 n}}{\rho(v)^{\frac 2 n}}v^2\chi_{\{v>\gamma_{k+1}\}}}_{L^1\left(L^1\right)(k+1)}\\
		\leq&\frac{1}{\rho(\gamma_{k+1})^{\frac 2 n}}\norm{\phi(v){v^{\frac 4 n}}\chi_{\{v>\gamma_{k+1}\}}}_{L^1\left(L^1\right)(k+1)}\\
		\lesssim&\frac 1 {\rho(\gamma_{\infty})^{\frac 2 n}}\norm{\phi(v)\chi_{\{v>\gamma_{k+1}\}}}_{L^1\left(L^{\frac n {n-2}}\right)(k+1)}\norm{v^{\frac 4 n}\chi_{\{v>\gamma_{k+1}\}}}_{L^\infty\left(L^{\frac 2 n}\right)(k+1)}\\
		=&\frac 1 {\rho(\gamma_{\infty})^{\frac 2 n}}\norm{\phi(v)\chi_{\{v>\gamma_{k+1}\}}}_{L^1\left(L^{\frac n {n-2}}\right)(k+1)}\norm{v\chi_{v^2>\gamma_{k+1}}}_{L^\infty\left(L^{1}\right)(k+1)}^{\frac 2 n}\\
		\lesssim&2^{3k\left(1+\frac{2}{n}\right)} W_k\left(\frac{W_k\alpha}{\rho(\gamma_\infty)}\right)^{\frac 2 n}
	\end{align*}
	
	In total, we have
	\begin{align*}
		W_{k+1}=&Y_{k+1}+\frac 1 \alpha Z_{k+1}\\
		\lesssim &2^{3k\left(1+\frac{2}{n}\right)}W_k\left(\frac{W_k\alpha}{\gamma_\infty^2}\right)^{\frac 2 n}+\frac 1 \alpha 2^{3k\left(1+\frac{2}{n}\right)}W_k\left(\frac{W_k\alpha}{\rho(\gamma_\infty)}\right)^{\frac 2 n}\\
		\lesssim &2^{3k\left(1+\frac{2}{n}\right)}W_k \max\left\{\left(\frac{W_k\alpha}{\gamma_\infty^2}\right)^{\frac 2 n},\left(\frac{W_k\alpha^{\frac{2-n}{n}}}{\rho(\gamma_\infty)}\right)^{\frac 2 n}\right\}\\
		=&2^{3k\left(1+\frac{2}{n}\right)}W_k\left(\frac{W_k}{\min\left\{\frac{\rho\left(\gamma_\infty\right)}{\alpha^{\frac{2-n}{2}}},\frac{\gamma_\infty^2}{\alpha}\right\}}\right)^{\frac{2}{n}}
	\end{align*}
	and the theorem follows as before from corollary \ref{cor} as we have $W_k\rightarrow 0$ for $\min\left\{\frac{\rho\left(\gamma_\infty\right)}{\alpha^{\frac{2-n}{2}}},\frac{\gamma_\infty^2}{\alpha}\right\}\sim W_0$:
	$$  \min\left\{\frac{\rho(v)}{\alpha^{\frac{2-n}{2}}},\frac{v^2}{\alpha}\right\}<\min\left\{\frac{\rho\left(\gamma_\infty\right)}{\alpha^{\frac{2-n}{2}}},\frac{\gamma_\infty^2}{\alpha}\right\}\sim W_0=\dashint \phi(v)+\frac{v^2}{\alpha}$$
	\enlargethispage{1cm}
\end{proof}

\newpage
\section{Appendix}
\begin{lemma}\label{interpol}
	Let $(\Omega_1,\mathcal{A}_1, \mu_1)$ and $(\Omega_2,\mathcal{A}_2, \mu_2)$ be measure spaces and denote the corresponding Lebesgue-Bochner-spaces by $L^p(L^q):=L^p(\Omega_1,L^q(\Omega_2,\mathbb{R}^m))$.\\
	\begin{enumerate} 
	 \item Let $p$,$p_1$, $p_2$, $q$, $q_1$, $q_2$ be real numbers greater than $1$ or infinity with $\frac{1}{p}=\frac{1}{p_1}+\frac{1}{p_2}$ and $\frac{1}{q}=\frac{1}{q_1}+\frac{1}{q_2}$ ($\frac 1 \infty=0$) and let $f\in L^{p_1}\left(L^{q_1}\right)$ and  $g\in L^{p_2}\left(L^{q_2}\right)$.\\
	 Then we have $fg\in L^{p}\left(L^{q}\right)$ and $\norm{fg}_{L^{p}\left(L^{q}\right)}\leq\norm{f}_{L^{p_1}\left(L^{q_1}\right)}\norm{g}_{L^{p_2}\left(L^{q_2}\right)}$
	 \item Let $p_0$, $p_1$, $q_0$, $q_1$ be real numbers greater than $1$ or infinity and let $f\in L^{p_0}\left(L^{q_1}\right)\cap L^{p_2}\left(L^{q_2}\right)$
	 Then for $\Theta\in [0,1]$ with $\frac 1 p=\frac \Theta {p_1}+\frac{1-\Theta}{p_0}$ and $\frac 1 1=\frac \Theta {q_1}+\frac{1-\Theta}{q_0}$ we have $f\in L^{p}\left(L^{q}\right)$.
	\end{enumerate}

\end{lemma}

\begin{proof}
	\begin{enumerate}
	 \item 
	 \begin{align*}
		\norm{fg}_{L^{p}\left(L^{q}\right)}&=\norm{\norm{fg}_{L^p}}_{L^q}\leq\norm{\norm{f}_{L^{p_1}}\norm{g}_{L^{p_2}}}_{L^q}\\
		&\leq\norm{\norm{f}_{L^{p_1}}}_{L^{q_1}}\norm{\norm{g}_{L^{p_2}}}_{L^{q_2}}=\norm{f}_{L^{p_1}\left(L^{q_1}\right)}\norm{g}_{L^{p_2}\left(L^{q_2}\right)}
	 \end{align*}
	 \item We use the H\"older-type estimate from above
	 \begin{align*}
		\norm{f}_{L^{p}\left(L^{q}\right)}&=\norm{f^{\Theta}f^{1-\Theta}}_{L^{p}\left(L^{q}\right)}\leq\norm{f^{\Theta}}_{L^{\frac{p_1}{\Theta}}\left(L^{\frac{q_1}{\Theta}}\right)}\norm{f^{1-\Theta}}_{L^{\frac{p_0}{1-\Theta}}\left(L^{\frac{q_0}{1-\Theta}}\right)}\\
		&=\norm{f}^{\Theta}_{L^{p_1}\left(L^{q_1}\right)}\norm{f}^{1-\Theta}_{L^{p_0}\left(L^{q_0}\right)}
	 \end{align*}

	\end{enumerate}

\end{proof}

\begin{lemma}\label{phiconvergence}
	Let $\phi$ be an N-Function with $\Delta_2(\phi)<\infty$. Then the following are equivalent:
	\begin{enumerate}
		\item $||f_n-f||_\phi\rightarrow 0$
		\item $\int \phi(|f_n-f|)\rightarrow 0$
	\end{enumerate}
	and those imply
	\begin{equation}\label{phiconvergence3}
		\left\vert\int \phi(|f_n|)-\int \phi(|f|)\right\vert\rightarrow 0
	\end{equation}
\end{lemma}

\begin{proof} (\cite{rao1991theory} Theorem 3.14.12)
	We show the theorem for $f=0$. For the general case we can just use $g_n=f_n-f$.\\
	(a)$\Rightarrow$(b): As we have $\norm{f_n}_\phi\rightarrow 0$ we have $\norm{f_n}_\phi\leq 1$ for $n$ large enough. This leads to
	\begin{equation*}
		\int \phi(|f_n|)=\int \phi\left(\frac{\norm{f_n}_\phi f_n}{\norm{f_n}_\phi}\right)\leq \norm{f_n}_\phi\int \phi\left(\frac{ f_n}{\norm{f_n}_\phi}\right)\leq \norm{f_n}_\phi\rightarrow 0
	\end{equation*}
	(b)$\Rightarrow$(a): Take $\epsilon>0$. Because of the $\Delta_2$-regularity of $\phi$ we have 
	$$\int \phi\left(\frac{|f_n|}{\epsilon}\right)\leq c_\epsilon \int \phi(|f_n|) $$
	As $\int \phi(|f_n|)\rightarrow 0$ there is an $N$ such that $\int \phi(|f_n|)\leq \frac{1}{c_\epsilon}$. But this means $\norm{f_n}_\phi\leq\epsilon$.\\
	For the last assertion it suffices to show that $\int \phi(|f+g|)\lesssim\int(\phi(|f|)+\phi(|g|))$. With the convexity and monotony of $\phi$ and the $\Delta_2$-condition we get
	$$\phi(|f+g|)\leq\phi(|f|+|g|)\leq\frac 1 2\left(\phi(2|f|)+\phi(2|g|)\right)\leq\frac {\Delta_2(\phi)}{2}\left(\phi(|f|)+\phi(|g|)\right)  $$	
\end{proof}

\begin{lemma}
	Let $\phi$ be a $\Delta_2$-regular N-function and $\Omega$ a bounded domain. Then the space of $C^\infty$-functions on $\Omega$ is dense in the Orlicz space $K^\phi(\Omega)$.
\end{lemma}
\begin{proof}	
	The proof is analogous to the $L^p$ case using that convergence in mean and convergence in norm are the same for a $\Delta_2$-regular $\phi$. At first, we show that simple functions are dense in $K^\phi$:\\
	Since $\phi(|\textbf{u}|)\in L^1$, we can find an increasing sequence of simple functions with $\int\phi(|\textbf{u}_n|)\nearrow \int\phi(|\textbf{u}|)$ by the definition of the Lebesgue integral. Since $\phi(|\textbf{u}_n|)\geq\phi(|\textbf{u}|)$ almost everywhere we have $\int |\phi(|\textbf{u}|)-\phi(|\textbf{u}_n|)|\rightarrow 0$ and can find a subsequence $\textbf{v}_n$ with $\textbf{v}_n\rightarrow \textbf{u}$ almost everywhere. By the monotone convergence theorem we therefore get $\int \phi(|\textbf{u}-\textbf{v}_n|)\rightarrow 0$.\\
	As we can approximate any simple function by a $C^\infty$-function in every $L^p$-space we can do so in $L^\phi$-spaces as well as we have $\phi(t)\lesssim (t^{\alpha_1}+t^{\alpha_2})\phi(1)$ (see \cite{1218603}) by taking a sequence of $C^\infty$-functions $u_n$ with (w.l.o.g. $\alpha_1>\alpha_2$) $\norm{\textbf{u}_n-\textbf{u}}_{\alpha_1}\rightarrow 0$. Then we get:
	\begin{align*}
		\int_{\Omega}\phi(|\textbf{u}_n-\textbf{u}|)&\lesssim \phi(1)\left(\norm{\textbf{u}_n-\textbf{u}}^{\alpha_1}_{\alpha_1}+\norm{\textbf{u}_n-\textbf{u}}_{\alpha_2}^{\alpha_2}\right)\\
		&\leq\phi(1)\left(\norm{\textbf{u}_n-\textbf{u}}^{\alpha_1}_{\alpha_1}+|\Omega|^{\frac{\alpha_1+\alpha_2}{\alpha_1\alpha_2}}\norm{\textbf{u}_n-\textbf{u}}_{\alpha_1}^{\frac{\alpha_2}{\alpha_1}}\right)\rightarrow 0
	\end{align*}
\end{proof}

\begin{lemma}
	Let $\phi$ be a $\Delta_2$-regular N-function and $\xi_\epsilon$ a standard mollifier. Denote by $\omega^\epsilon$ the outer parallel set of $\omega\Subset\Omega$. Then  for $\omega^{\epsilon}\Subset\Omega$ we have:
	$$\int_\omega \phi(|\textbf{u}_\epsilon|) \leq \int_{\omega^\epsilon}\phi(|\textbf{u}|)$$. 
\end{lemma}
\begin{proof}
	For $L^1_{\text{loc}}$-functions $\textbf{u}$ we get using $\int \xi=1$:
	\begin{align*}
		&\int_\omega\int_{\omega^\epsilon}\xi_\epsilon(y-x)|\textbf{u}(y)|\dz\dx\\
		\leq&\int_{\omega^\epsilon}\int_{\omega\cap B_\epsilon(y)}\xi_\epsilon(y-x)\dx|\textbf{u}(y)|\dy\leq\int_{\omega^\epsilon}|\textbf{u}(y)|\dy
	\end{align*}
	We now define an $x$-dependent measure via $\text{d}\mu_x=\xi_\epsilon(y-x)\dy$ and note that $\int_{\omega^\epsilon} \text{d}\mu_x=1$. Using Jensen's inequality and the above result with the $L^1_{\text{loc}}$-function $\phi(|\textbf{u}|)$ we get:
	\begin{align*}
		&\int_{\omega}\phi\left(\left\vert\int_{\omega^\epsilon}\xi_\epsilon(y-x)\textbf{u}(y)\dy\right\vert\right)\dx\leq\int_{\omega}\phi\left(\int_{\omega^\epsilon} |\textbf{u}(y)|\text{d}\mu_x\right)\dx\\
		\leq&\int_{\omega}\int_{\omega^\epsilon} \phi\left(|\textbf{u}(y)|\right)\text{d}\mu_x\dx\leq\int_{\omega^\epsilon}\phi(|\textbf{u}(y)|)\dy
	\end{align*}
\end{proof}

\begin{lemma}
	Let $\phi$ be a $\Delta_2$-regular N-function and $\xi_\epsilon$ a standard mollifier. Then for every $\textbf{u}\in L^\phi_{\text{loc}}$ we have $\textbf{u}_\epsilon:=\textbf{u}\ast\xi_\epsilon\rightarrow \textbf{u}$ as $\epsilon\rightarrow 0$.
\end{lemma}
\begin{proof}
	Take an $\omega\Subset\Omega$. We know that for smooth functions $\textbf{v}$ we have $\textbf{v}_\epsilon\rightarrow \textbf{v}$ locally uniform and therefore also in $\phi$-mean and in the $\phi$-Luxemburg norm. Let $\delta>0$ be fixed. For $\textbf{u}\in L^\phi$ we chose a $\textbf{v}\in C^\infty$ such that $\norm{\textbf{v}-\textbf{u}}_{\phi,\omega^{\epsilon_0}}\leq\frac \delta 3$ for some $\epsilon_0>0$. We also chose $0<\epsilon<\epsilon_0$ small enoug that $\norm{\textbf{v}_\epsilon-\textbf{v}}_{\phi,\omega}\leq\frac\delta 3$holds. Then we get:
	\begin{align*}
		&\norm{\textbf{u}-\textbf{u}_\epsilon}_{\phi,\omega}\leq\norm{\textbf{u}-\textbf{v}}_{\phi,\omega}+\norm{\textbf{v}-\textbf{v}_\epsilon}_{\phi,\omega}+\norm{\textbf{v}_\epsilon-\textbf{u}_\epsilon}_{\phi,\omega}\\
		\leq&\norm{\textbf{u}-\textbf{v}}_{\phi,\omega}+\norm{\textbf{v}-\textbf{v}_\epsilon}_{\phi,\omega}+\norm{\textbf{v}-\textbf{u}}_{\phi,\omega^{\epsilon_0}}<\delta
	\end{align*}
\end{proof}

\begin{lemma}\label{convexintegral}(cf \cite{diening2008fractional} Lemma 20)
	Let $\phi$ be an N-function with $\Delta_2(\{\phi,\phi^\ast\})<\infty$ and $[\textbf{P},\textbf{Q}]_s=s\textbf{P}+(1-s)\textbf{Q}$ as before. Then we have 
	$$\int_0^1\frac{\phi'(|[\textbf{P},\textbf{Q}]_s| )}{|[\textbf{P},\textbf{Q}]_s|}\ds\sim\frac{\phi'(|\textbf{P}|+|\textbf{Q}|)}{|\textbf{P}|+|\textbf{Q}|}$$
\end{lemma}
\begin{proof}
	Because of $\Delta_2({\phi^\ast})<\infty$ we have (cf \cite{kokilashvili1991weighted} Lemmas 1.2.2 and 1.2.3) a $\theta\in(0,1)$ and an N-function $\rho$ such that $\phi^\theta\sim\rho$ with $\Delta_2(\{\rho,\rho^\ast\})<\infty$ and $\rho'(t)t\sim\rho(t)$ and therefore $\phi'(t)\sim\frac{\phi(t)}{t}\sim\frac{\rho(t)^{\frac 1 \Theta}}{t}\sim\rho'(t)t^{\frac 1 \Theta-1}$. This gives
	\begin{align*}
		\int_0^1\frac{\phi'(|[\textbf{P},\textbf{Q}]_s|)}{|[\textbf{P},\textbf{Q}]_s|}\ds&\lesssim\int_0^1\rho'(|[\textbf{P},\textbf{Q}]_s|)^{\frac 1 \theta}|[\textbf{P},\textbf{Q}]_s|^{\frac 1 \theta -2}\ds\\
		&\leq\left(\rho'(|\textbf{P}|+|\textbf{Q}|)\right)^{\frac 1 \theta}\int_0^1|[\textbf{P},\textbf{Q}]_s|^{\frac 1 \theta -2}\ds\\
		&\lesssim \left(\rho'(|\textbf{P}|+|\textbf{Q}|)\right)^{\frac 1 \theta}(|\textbf{P}|+|\textbf{Q}|)^{\frac 1 \theta -2}\\
		&=\frac{\left(|\textbf{P}|+|\textbf{Q}|)(\rho'(|\textbf{P}|+|\textbf{S}|)\right)^{\frac 1 \theta}}{(|\textbf{P}|+|\textbf{Q}|)^2}\\
		&\sim \frac{\phi'(|\textbf{P}|+|\textbf{Q}|)}{|\textbf{P}|+|\textbf{Q}|}
	\end{align*}
	where we used $(|\textbf{P}|+|\textbf{Q}|)\sim\int_0^1|[\textbf{P},\textbf{Q}]_s|\ds $.\\
	For the other direction we see using $\phi(t)\sim\phi'(t)t$, $|[\textbf{P},\textbf{Q}]_s|\leq|\textbf{P}|+|\textbf{Q}|$ and Jensen's inequality that
	\begin{align*}
		\int_0^1\frac{\phi'(|[\textbf{P},\textbf{Q}]_s|)}{|[\textbf{P},\textbf{Q}]_s|}\ds\gtrsim\int_0^1\frac{\phi(|[\textbf{P},\textbf{Q}]_s|)}{(|\textbf{P}|+|\textbf{Q}|)^2}\geq\frac{\phi\left(\int_0^1|[\textbf{P},\textbf{Q}]_s|\ds\right)}{(|\textbf{P}|+|\textbf{Q}|)^2}
	\end{align*}
	We now use that $\int_0^1|[\textbf{P},\textbf{Q}]_s|\ds\gtrsim c(|\textbf{P}|+|\textbf{Q}|)$ (see for example \cite{acerbi1989regularity}) and use the $\Delta_2$ regularity of $\phi$:
		\begin{align*}
			\int_0^1\frac{\phi'(|[\textbf{P},\textbf{Q}]_s|)}{|[\textbf{P},\textbf{Q}]_s|}\ds\gtrsim\frac{\phi\left(|\textbf{P}|+|\textbf{Q}|\right)}{(|\textbf{P}|+|\textbf{Q}|)^2}\sim\frac{\phi'\left(|\textbf{P}|+|\textbf{Q}|\right)}{|\textbf{P}|+|\textbf{Q}|}
		\end{align*}
\end{proof}

\begin{lemma}\label{psi}
	Let $\phi$ be an N-function satisfying assumption \ref{mainassumption}. Then the associated N-function $\psi$ defined via $\psi'(t)=\sqrt{t\phi'(t)}$ also satisfies assumption \ref{mainassumption} and we have $\psi''(t)\sim\sqrt{\phi''(t)}$
\end{lemma}

\begin{proof} We get
	\begin{align*}
		t\psi''(t)=\frac{1}{2\sqrt{t\phi'(t)}}\left(\phi'(t)+t\phi''(t)\right)\sim\sqrt{t\phi'(t)}=\psi'(t)
	\end{align*}
	and use this to show
	\begin{equation*}
		t\psi''(t)\sim\psi'(t)=\sqrt{t\phi'(t)}\sim\sqrt{t^2\phi''(t)}=t\sqrt{\phi''(t)}
	\end{equation*}
\end{proof}

\begin{lemma}\label{shiftedDelta2}
\sloppy	Let $\phi$ be an N-function with $\Delta_2(\{\phi,\phi^\ast\})<\infty$. Then $\Delta_2(\{\phi_\lambda,\phi_\lambda^\ast\}_{\lambda\geq0})$ is bounded uniformly in $\lambda$.
\end{lemma}

\begin{proof}(cf \cite{diening2008fractional} Lemma 23)
	As we have $\phi'_\lambda(t)t\sim\phi_\lambda(t)$ uniformly in $\lambda$ and $\phi'(2t)\sim\phi'(t)$ and $\lambda+2t\sim\lambda+t$ we get
	\begin{align*}
		\phi'_\lambda(2t)=\frac{\phi'(\lambda+2t)}{\lambda+2t}2t\sim\frac{\phi'(\lambda+t)}{\lambda+t}t=\phi'_\lambda(t)
	\end{align*}
	and this proves the claim for $\phi_\lambda$. The proof for $\phi_\lambda^\ast$ is analogous.
\end{proof}

\begin{lemma}\label{rausziehen}
	Let $\phi$ be an N-function with $\Delta_2(\{\phi,\phi^\ast\})<\infty$. Then we have an $\epsilon>0$ depending only on $\Delta_2(\{\phi,\phi^\ast\})$ such that $\phi_\lambda(kt)\lesssim k^{1+\epsilon}\phi_\lambda(t)$ holds for all $0\leq k \leq 1$.
\end{lemma}

\begin{proof}(see Lemma 31 in \cite{diening2008fractional})
	Like in the proof of \ref{convexintegral} we have an N-function $\rho$ with $\phi^\Theta\sim\rho$ for a $0<\Theta<1$. Then we get uniformly in $t$ and $k$:
	\begin{equation*}
		\phi(kt)\sim\left(\rho(kt)\right)^{\frac 1 \Theta}\sim k^{\frac 1 \Theta}\phi(t)
	\end{equation*}
	This shows the claim for $\lambda=0$ with $\epsilon=\frac 1 \Theta -1$. As we have $\Delta_2(\{\phi_\lambda,\phi_\lambda^\ast\}_{\lambda\geq0})$ from lemma \ref{shiftedDelta2} the proof for $\phi_\lambda$ is analogous.
\end{proof}

\begin{lemma}\label{shiftedrausziehen}
	Let $\phi$ be an N-function with $\Delta_2(\{\phi,\phi^\ast\})<\infty$. Then we have $\phi_\lambda(\lambda_k)\sim k^2\phi(\lambda)$ uniformly in $0\leq k \leq 1$
\end{lemma}

\begin{proof}
	We note that $k\lambda+\lambda\sim\lambda$ and $\phi'(ct)\sim\phi(t)$ because of the $\Delta_2$ condition and estimate
	\begin{equation*}
		\phi_\lambda(k\lambda)\sim k\lambda\phi'_\lambda(k\lambda)=k^2\lambda^2\frac{\phi'(k\lambda+\lambda)}{k\lambda+\lambda}\sim k^2\lambda\phi'(\lambda)\sim k^2\phi(\lambda)
	\end{equation*}
\end{proof}

\begin{theorem}\label{stattheorem1}
	Let $\phi$ be an N-function satisfying assumption \ref{mainassumption} with $\Delta_2(\{\phi,\phi^\ast\})<\infty$ and $\textbf{u}\in W^{1,\phi}_{\text{loc}}(\Omega)$ be a local weak solution to $\Delta_\phi \textbf{u}=0$ on a domain $\Omega\subset\mathbb{R}^n$. Then we have $\textbf{\textbf{V}}(\nabla \textbf{u})\in W_{\text{loc}}^{1,2}(\Omega)$.
\end{theorem}	

We proceed like in \cite{diening2008fractional} and begin by showing the following
\begin{theorem}\label{stattheorem2} 
	Let $\textbf{u}$ be a local weak solution of $\Delta_\phi \textbf{u}=0$ on $\Omega$. For a cube $Q$ with side-length $R$ and $5Q\Subset\Omega$ we have the inequality:
	\begin{equation}\label{stattheorem2equation}
		\dashint_Q|\tau_h \textbf{V}(\nabla \textbf{u})|^2\dx\lesssim\frac{|h|^2}{R^2}\dashint_{5Q}|\textbf{V}(\nabla \textbf{u})|^2\dx
	\end{equation}
\end{theorem}
The proof is split into two parts
\begin{lemma}\label{statlemma1}
	Let $u$ be a local weak solution of $\Delta_\phi \textbf{u}=0$ on $\Omega$. For a cube $Q$ with side-length $R$ and $4Q\Subset\Omega$ we have the inequality:
	\begin{equation}
		\dashint_0^\lambda\int_{Q}|\tau_s V(\nabla \textbf{u})|^2\dx\lesssim\epsilon\dashint_0^\lambda\int_{4Q}|\tau_s \textbf{V}(\nabla \textbf{u})|^2\dx\text{d}\lambda+c_\epsilon\frac{\lambda^2}{R^2}\int_{4Q}\phi(|\nabla \textbf{u}|)\dx
	\end{equation}
\end{lemma}
\begin{proof}	
	We take equation \ref{all} on $2Q$ and $f\equiv 1$, multiply with $h^2$ and take the $C^\infty$ function $\eta$ with $\chi_Q\leq\eta\leq\chi_{2Q}$ and $|\nabla \eta|<R^{-1}$. We get
	
	\begin{align}
		0&=\langle \textbf{A}(\nabla \textbf{u}),\nabla (\tau_{j,-h}(\tau_{j,h}\textbf{u}\eta^q))\rangle=	\langle\tau_{j,h}\textbf{A}(\nabla \textbf{u}),\nabla(\delta_{j,h}\textbf{u}\eta^q)\nonumber\\
	&=\langle \delta_{j,h}\textbf{A}(\nabla \textbf{u}),\delta_{j,h}\nabla \textbf{u}\eta^q+\delta_{j,h}\textbf{u} q\eta^{q-1}\nabla \eta\rangle=\text{I}+\text{II}\label{all2}
	\end{align}
	Like in \ref{I} we get 
	\begin{equation}\label{I2}
		\text{I}\sim\dashint_{2Q}|\tau_{j,h}\textbf{V}(\nabla \textbf{u})|^2\eta^q\dx\geq \dashint_{Q}|\tau_{j,h}\textbf{V}(\nabla \textbf{u})|\dx
	\end{equation}
	and in analogy to \ref{IIIbeg} we get
	\begin{align}
		\text{II}\lesssim&\int_{2Q}\dashint_0^h\eta^{q-1}\phi'_{|\nabla \textbf{u}|}(|\tau_{j,h}\nabla \textbf{u}|)|\nabla \textbf{u}\circ T_{se_j}|\,h|\nabla \eta|\ds\nonumber\\
		\leq&\int_{2Q}\dashint_0^h \eta^{q-1}\frac{h}{R}\phi'_{|\nabla \textbf{u}|}(|\tau_{j,h}\nabla \textbf{u}|)|\nabla \textbf{u}\circ T_{se_j}|\ds\label{II2}
	\end{align}
	Replacing the factor $h$ by $\lambda$ and and $|\nabla\eta|$ by $R^{-1}$ in \ref{phiestimate} we get the inequality
	\begin{align*}
		&\eta^{q-1}\phi'_{|\nabla \textbf{u}|}(|\tau_h\nabla \textbf{u}|)|\nabla \textbf{u}\circ T_{se_j}|\,\frac{\lambda}{R}\nonumber\\
		\lesssim&\epsilon\eta^q|\tau_{j,h-s}\textbf{V}(\nabla \textbf{u})\circ T_{se_j}|^2+\epsilon\eta^q|\tau_{j,s} \textbf{V}(\nabla \textbf{u})|^2+c_\epsilon \frac{\lambda^2}{R^2}\phi\left(|\nabla \textbf{u}\circ T_{se_j}|\right)
	\end{align*}
	Putting this in \ref{II2} we get
	\begin{align}
		\text{II}\leq&\epsilon\frac{h}{\lambda}\int_{2Q}\dashint_0^h \eta^q|\tau_{j,h-s}\textbf{V}(\nabla \textbf{u})\circ T_{se_j}|^2\ds\dx\nonumber\\
		+&\epsilon\frac{h}{\lambda}\int_{2Q}\dashint_0^h\eta^q|\tau_{j,s} \textbf{V}(\nabla \textbf{u})|^2\ds\dx+c_\epsilon \frac{\lambda^2}{R^2}\int_{2Q}\dashint_0^h\phi\left(|\nabla \textbf{u}\circ T_{se_j}|\right)\ds\dx\nonumber\\
		\leq&\epsilon\frac{h}{\lambda}\int_{2Q}\dashint_0^h |\tau_{j,h-s}\textbf{V}(\nabla \textbf{u})\circ T_{se_j}|^2\dx\nonumber\\
		+&\epsilon\frac{h}{\lambda}\int_{2Q}\dashint_0^h|\tau_{j,s} \textbf{V}(\nabla \textbf{u})|^2\ds\dx+c_\epsilon \frac{\lambda^2}{R^2}\int_{2Q}\dashint_0^h\phi\left(|\nabla \textbf{u}\circ T_{se_j}|\right)\ds\dx\label{II3}
	\end{align}
	We now note for a general $f\in L^1_{\text{loc}}$ and $s<R$
	
	\begin{align*}
		&\int_{2Q}\dashint_0^h |(f\circ T_s)(x)|\ds\dx\\
		=&\dashint_0^h \int_{\mathbb{R}^n}\chi_{2Q}(x)|(f\circ T_s)(x)|\dx\ds\\
		=&\dashint_0^h \int_{\mathbb{R}^n}\underbrace{(\chi_{2Q}\circ T_{-s})}_{\leq\chi_{4Q}(x)}(x)|f(x)|\dx\ds\\
		\leq&\int_{4Q}\dashint_0^h |(f)(x)|\ds\dx\\
	\end{align*}
	and
	\begin{align*}
		&\int_{2Q}\dashint_0^h|\left(\tau_{h-s}f\circ T_s\right)(x)|\ds\dx\\
		=&\int_{2Q}\dashint_0^h|\left(\tau_{s}f\circ T_{h-s}\right)(x)|\ds\dx\\
		=&\dashint_0^h\int_{\mathbb{R}^n}\underbrace{(\chi_{2Q}\circ T_{s-h})(x)}_{\leq\chi_{4Q}(x)}|\left(\tau_{s}f\right)(x)|\ds\dx\\
		\leq&\int_{4Q}\dashint_0^h|\left(\tau_{s}f\right)(x)|\ds\dx\\
	\end{align*}
	Putting those 2 estimates in \ref{II3} and putting it with \ref{I2} in \ref{all2} we get
	\begin{align}
		\dashint_{Q}|\tau_{j,h}\textbf{V}(\nabla \textbf{u})|^2\dx\leq\epsilon\frac{h}{\lambda}\int_{4Q}\dashint_0^h |\tau_{j,s}\textbf{V}(\nabla \textbf{u})|^2\dx+c_\epsilon \frac{\lambda^2}{R^2}\int_{4Q}\left(|\nabla \textbf{u}|\right)\dx\label{II4}
	\end{align}
	 We note that we get for any $L^1$-function $g$:
	\begin{align*}
		&\dashint_0^\lambda\frac{h}{\lambda}\dashint_0^h|g(s)|\ds\text{d}h=\frac{1}{\lambda^2}\int_0^1\int_0^1\chi_{(0,h)}(s)\chi_{(0,\lambda)}(h) |g(s)|\ds\text{d}h\\
		=&\frac{1}{\lambda^2}\int_0^1\int_0^1\chi_{(s,\lambda)}(h)\chi_{(0,\lambda)}(s) |g(s)|\ds\text{d}h=\dashint_0^\lambda\frac 1 \lambda \int_s^\lambda\text{d}h|g(s)|\ds\\
		\leq&\dashint_0^\lambda|g(s)|\ds
	\end{align*}
	Integrating \ref{II4} via $\dashint_0^\lambda\text{d}h$ proves lemma \ref{statlemma1}.
\end{proof}
To conclude the proof of theorem \ref{stattheorem2} we need a lemma from \cite{diening2008fractional}:
\begin{lemma}\label{Giaquinta}
	Let $\gamma_1$, $\gamma_2$ functions such that $\gamma_i(R,h)$ is non decreasing in $h$ and $\frac{h}{R}$. Let $f\in L^2_{\text{loc}}(\Omega)$ and $g_i\in L^2_{\text{loc}}(\Omega)$ be functions such that the following statement is true: For every $\epsilon>0$ there is a $c_\epsilon>0$ such that for every cube $Q$ with side length $R$ and $4Q\Subset\Omega$ and every $0<h<R$ holds:
	\begin{align}
		\dashint_0^\lambda\int_{Q}|\tau_s f|^2\dx\lesssim\epsilon\dashint_0^\lambda\int_{4Q}|\tau_s f|^2\dx\ds+c_\epsilon\sum_{i=1}^2\gamma_i(R,h)\int_{4Q}g_i\dx
	\end{align}
	Then there exist constants $N_2(n)$ and $c$ such that for every $0<h<\frac {R_0} {10}$ and every cube $Q_0$ with $5Q_0\Subset\Omega$ holds
	\begin{equation}
		\int_{Q_0}|\tau_s f|^2\dx\lesssim c\sum_{i=1}^2\gamma_i(R,h)\int_{5Q_0}g_i\dx
	\end{equation}
\end{lemma}
\begin{proof}
	\cite{diening2008fractional} Lemma 13.
\end{proof}

We are now able to prove theorem \ref{stattheorem2}.
\begin{proof}[Proof of theorem \ref{stattheorem2}]
	From lemma \ref{statlemma1} we know that the assumptions of lemma \ref{Giaquinta} are fulfilled with $f=\textbf{V}(\nabla \textbf{u}) $, $\gamma_1(R,h)=\frac{h^2}{R^2}$, $\gamma_2=0$ and $g_1=\phi(|\nabla \textbf{u}|)$. To conclude the proof we note $\gamma_1(N_2R,N_2h)=\gamma_1(R,h)$ 
\end{proof}

\begin{proof}[Proof of Theorem \ref{stattheorem1}]
	We divide equation \ref{stattheorem2equation} by $h^2$ and get
	\begin{equation*}
		\dashint_Q|\delta_h \textbf{V}(\nabla \textbf{u})|^2\dx\lesssim\frac{1}{R^2}\dashint_{5Q}|\textbf{V}(\nabla \textbf{u})|^2\dx<\infty
	\end{equation*}
	This implies the existence of $\nabla \textbf{V}(\nabla \textbf{u})\in L^2(Q)$ for every Cube $Q$ with $5Q\Subset\Omega$. For any other $\omega\Subset\Omega$ we denote by $R=\text{dist}(\omega,\partial\Omega)$. Take the open covering $\omega\subset\cap_{x\in\omega}Q_{\frac{R}{6}}(x)\subset \Omega$ since $\omega$ is compact we have a finite subcovering of cubes $Q_i:=Q_{\frac{R}{6}}(x_i)$, $i=1,...,N$, with $5Q_i\Subset\Omega$. Therefore we have 
	\begin{equation*}
		\dashint_\omega|\delta_h \textbf{V}(\nabla \textbf{u})|^2\dx\lesssim\frac{1}{R^2}\sum_{i=1}^N\dashint_{5Q_i}|\textbf{V}(\nabla \textbf{u})|^2\dx<\infty
	\end{equation*}
\end{proof}

\begin{theorem}\label{instatheorem1}
	Let $\phi$ be an N-function satisfying assumption \ref{mainassumption} and $\textbf{u}\in L^\phi_{\text{loc}}(J\times\Omega,\mathbb{R}^m)\cap C_{\text{loc}}(J,L^2(\Omega,\mathbb{R}^m))$ be a local weak solution to $\Delta_\phi \textbf{u}=\textbf{u}_t$ on a cylindric domain $J\times\Omega\subset\mathbb{R}^{1+n}$ with $v:=|\nabla \textbf{u}|\in L_{\text{loc}}^2(J\times \Omega)\cap L_{\text{loc}}^\phi(J\times \Omega)$. Then we have  $\textbf{V}(\nabla \textbf{u})\in L^2_{\text{loc}}(I,W_{\text{loc}}^{1,2}(\Omega,\mathbb{R}^m))$.	
\end{theorem}

In analogy to the elliptic case we divide the proof.
\begin{lemma}\label{instatlemma1}
	Let $\phi$ be an N-function satisfying assumption \ref{mainassumption} and $\textbf{u}\in L^\phi_{\text{loc}}(J\times\Omega,\mathbb{R}^m)\cap C_{\text{loc}}(J,L^2(\Omega,\mathbb{R}^m))$ be a local weak solution to $\Delta_\phi \textbf{u}=\partial_t{\textbf{u}}$ on a cylindric domain $J\times\Omega\subset\mathbb{R}^{1+n}$ with with $v:=|\nabla \textbf{u}|\in L_{\text{loc}}^2(J\times\Omega)\cap L_{\text{loc}}^\phi(J\times\Omega)$. Then for every space time cube $Q$ of sidelength $R$ with $4Q\Subset J\times\Omega$ and every $\lambda<R$ we have
	\begin{align}\label{instatlemma1equation}
		&\dashint_0^\lambda\int_{Q}|\tau_s \textbf{V}(\nabla \textbf{u})|^2\dz\leq\epsilon\dashint_0^\lambda\int_{4Q}|\tau_s \textbf{V}(\nabla \textbf{u})|^2\dx\text{d}s\nonumber\\
		+c_\epsilon&\left(\frac{\lambda^2}{R^2}\int_{4Q}\phi(|\nabla \textbf{u}|)\dz+\frac{\lambda^2}{R}\int_{4Q}|\nabla \textbf{u}|^2\dz\right)
	\end{align}
\end{lemma}
\begin{proof}
	We multiply the inequality \ref{instatdiscretenergy} on $2Q$ by $h^2$, set $f\equiv 1$ and discard $\text{II'}$:	
	\begin{align*}
		\int_{2Q'}\tau_{h,j}\textbf{A}(\nabla \textbf{u})\nabla (\tau_{h,j}\textbf{u}\rho(t)\eta^q)\dz\leq h^2\int_{2Q}  H(|\delta_h \textbf{u}|)\partial_t\left(\eta^q\right)\dz
	\end{align*}
	We now take $\eta\in C_0^\infty$ such that $\chi_Q\leq\eta\leq\chi_{2Q}$, $|\nabla \eta|\leq R^{-1}$ and $|\partial_t \eta|\leq R^{-1} $ and get

\begin{align}
		\text{I''}:=\int_{Q}\tau_{h,j}\textbf{A}(\nabla \textbf{u})\nabla (\tau_{h,j}\textbf{u})\dz\leq R^{-1}\int_{2Q}  |\tau_h \textbf{u}|^2\dz=:\text{II''}\label{instatdiscretenergy2}
	\end{align}
Since $\textbf{u}\in L^2(W^{1,2})$ we have $\frac 1 {h^2} \int_{2Q}  |\tau_h \textbf{u}|^2\dz\rightarrow\int_{2Q}|\nabla \textbf{u}|^2\dz$ and therefore for every $\lambda>h$

$$  \text{II''}\leq \frac 2 R {h^2}\int_{2Q}|\nabla \textbf{u}|^2\dz\leq 2{\lambda^2}\int_{4Q}|\nabla \textbf{u}|^2\dz $$

We then handle $\text{I''}$ like in lemma \ref{statlemma1} and take $\max \{c_\epsilon,2\}$ as our new $c_\epsilon$ to get the result of lemma \ref{instatlemma1}
\end{proof}

\begin{proof}[Proof of theorem \ref{instatheorem1}]
	We use the Giaquinta-Modica type lemma \ref{Giaquinta} with $\gamma_1(R,\lambda)=\frac{\lambda^2}{R^2}$, $\gamma_2=\frac{\lambda^2}{R}$, $g_1=\phi(|\nabla \textbf{u}|)$ and $g_2=|\nabla \textbf{u}|^2$. We get
	
	\begin{equation*}
		\dashint_Q|\tau_\lambda \textbf{V}(\nabla \textbf{u})|^2\dz\leq c\left(\frac{\lambda^2}{R^2}\dashint_{5_Q}\phi(|\nabla \textbf{u}|)+\frac{\lambda^2}{R}\dashint_{5Q}|\nabla \textbf{u}|^2\right)
	\end{equation*}
	Dividing this by $\lambda^2$ leads to 
	\begin{equation*}
		\dashint_Q|\delta_\lambda \textbf{V}(\nabla \textbf{u})|^2\dz\leq c\left(\frac{1}{R^2}\dashint_{5_Q}\phi(|\nabla \textbf{u}|)+\frac{1}{R}\dashint_{5Q}|\nabla \textbf{u}|^2\right)<\infty
	\end{equation*}
	which implies $\textbf{V}(\nabla \textbf{u})\in W^{1,2}(Q)$ for every cube $Q$ with $5Q\Subset\Omega$. The same simple covering argument as in the elliptic case leads to $\textbf{V}(\nabla \textbf{u})\in W^{1,2}_{\text{loc}}(\Omega)$
\end{proof}

\bibliographystyle{unsrt}

\end{document}